\documentclass[pdflatex]{sn-jnl}%
\usepackage{graphicx}%
\usepackage{multirow}%
\usepackage{amsmath,amssymb,amsfonts}%
\usepackage{amsthm}%
\usepackage{bm} 
\usepackage{mathrsfs}%
\usepackage[title]{appendix}%
\usepackage{xcolor}%
\usepackage{textcomp}%
\usepackage{manyfoot}%
\usepackage{booktabs}%
\usepackage{algorithm}%
\usepackage{algorithmicx}%
\usepackage{algpseudocode}%
\usepackage{listings}%
\usepackage{tikz}
\usepackage{cite}

\numberwithin{equation}{section}  

\theoremstyle{thmstyleone}
\newtheorem{theorem}{Theorem}[section] 

\newtheorem{remark}{Remark}[section]
\newtheorem{lemma}{Lemma}[section]

\usepackage{chngcntr}
\counterwithin{figure}{section}
\counterwithin{table}{section}

\begin{document}
\title[Article Title]{Local Projection Stabilization Methods for  $\bm{H}(\mathrm{curl})$ and  $\bm H(\mathrm{div})$ Advection Problems}

\author[1]{\fnm{Yangfan} \sur{Luo}}\email{luoyf@stu.pku.edu.cn}

\author[2]{\fnm{Jindong} \sur{Wang}}\email{jzw6472@psu.edu}

\author*[1]{\fnm{Shuonan}\sur{Wu}}\email{snwu@math.pku.edu.cn}
\affil[1]{\orgdiv{School of Mathematical Sciences}, \orgname{Peking University}, \orgaddress{ \city{Beijing}, \postcode{100871},  \country{China}}}

\affil[2]{\orgdiv{Department of Mathematics}, \orgname{The Pennsylvania State University}, \orgaddress{\city{Univerisity Park}, \postcode{16802}, \state{PA}, \country{USA}}}

\affil*[1]{\orgdiv{School of Mathematical Sciences}, \orgname{Peking University}, \orgaddress{ \city{Beijing}, \postcode{100871},  \country{China}}}

\abstract{We devise local projection stabilization (LPS) methods for advection problems in the $\bm{H}$(curl) 
and $\bm{H}$(div) spaces, employing conforming finite element spaces of arbitrary order within a unified framework. The key ingredient is a local inf-sup condition, enabled by enriching the approximation space with appropriate $\bm{H}$(d) bubble functions (with d = curl or div). This enrichment allows for the construction of modified interpolation operators, which are crucial for establishing optimal a priori error estimates in the energy norm. Numerical examples are presented to verify both the theoretical results and the stabilization properties of the proposed method.
}


\maketitle

\section{Introduction}\label{sec1}
Consider the following advection-diffusion equation for the magnetic vector potential
\begin{equation}
    \label{eqn curl time-dep}
 \nabla \times (\varepsilon \, \nabla \times \bm{u}) + (\nabla \times \bm{u}) \times \bm{\beta} + \nabla (\bm{\beta} \cdot \bm{u}) + \gamma \bm{u} = \bm{f},
\end{equation}
which describes the magnetic field in a conducting fluid moving with velocity $\bm{\beta}$, where $\varepsilon$ denotes the magnetic diffusion coefficient. This magnetic advection-diffusion problem plays a central role in many scientific and engineering applications, particularly in magnetohydrodynamics (MHD) \cite{gerbeau2006mathematical}.

Alternatively, one can model the system using the magnetic induction field (again denoted by $\bm u$), leading to the following formulation:
\begin{equation}
\label{eqn div time-dep}
 - \nabla (\varepsilon \, \nabla \cdot \bm{u})+ \bm{\beta} (\nabla \cdot \bm{u}) + \nabla \times (\bm{u} \times \bm{\beta})  + \gamma \bm{u} = \bm{f}.
\end{equation}
Equations \eqref{eqn curl time-dep} and \eqref{eqn div time-dep} both represent generalized advection-diffusion problems and can be interpreted as advection-diffusion equations corresponding to different differential forms \cite{heumann2010eulerian,heumann2016stabilized}.

One of the main numerical challenges for \eqref{eqn curl time-dep} and \eqref{eqn div time-dep} is the stability with respect to $\varepsilon$ for the advection-dominated case, i.e., when $\varepsilon \ll |\bm{\beta}|$. In this case, standard numerical methods may suffer from spurious oscillations or loss of accuracy. This difficulty is particularly pronounced and has been extensively studied in the context of the closely related scalar advection-diffusion problem
\begin{equation}\label{eq:scalar_advection_diffusion}
-\nabla \cdot (\varepsilon\nabla u) + \bm{\beta} \cdot \nabla u + \gamma u = f. 
\end{equation}

To address this challenge in the scalar case, a wide range of numerical methods has been developed. One broad class of approaches for the scalar advection-diffusion equation \eqref{eq:scalar_advection_diffusion} focuses on enhancing stability through upwinding or streamline-based stabilization techniques. This class includes methods such as the stabilized discontinuous Galerkin (DG) method \cite{houston2002discontinuous, brezzi2004discontinuous, ayuso2009discontinuous}, the streamline upwind/Petrov-Galerkin (SUPG) method \cite{mizukami1985petrov, hughes1979finite, franca1992stabilized, chen2005optimal, burman2010consistent}, the Galerkin/least-squares finite element method \cite{hughes1989new}, bubble function stabilization \cite{brezzi1994choosing, brezzi1998applications, brezzi1998further, brezzi1999priori, franca2002stability}, the continuous interior penalty (CIP) method \cite{burman2005unified, burman2007continuous, burman2009weighted}, and the local projection stabilization (LPS) method \cite{braack2006local, matthies2007unified, knobloch2010generalization}.

Recently, the vector advection-diffusion problems \eqref{eqn curl time-dep} and \eqref{eqn div time-dep} have attracted growing attention from the numerical analysis community. Exponentially-fitted methods, which provide intrinsic stabilization, have been proposed in \cite{wu2020simplex, wang2023exponentially}. For upwinding-based stabilization, Heumann and Hiptmair applied a specialized class of upwind methods for vector-valued problems \cite{heumann2010eulerian, heumann2013stabilized}. In addition, primal discontinuous Galerkin and hybridizable discontinuous Galerkin methods for the magnetic advection-diffusion problem \eqref{eqn curl time-dep} have been developed in \cite{wang2024discontinuous, wang2024hybridizable}. Most of the existing stabilization approaches rely on non-conforming finite element spaces. However, to the best of our knowledge, there is currently no work in the literature that employs conforming finite element spaces combined with local projection stabilization involving symmetric stabilizing terms.

The local projection stabilization (LPS) technique was originally developed to enable the stable discretization of the Stokes problem using equal-order finite elements \cite{becker2001finite}. It has since been successfully extended to various applications, including transport equations \cite{becker2004two}, Oseen equations \cite{braack2006local}, and advection-diffusion-reaction problems \cite{ganesan2010stabilization, knobloch2009local, knobloch2010generalization}. 
A key component of the LPS method is the satisfaction of a local inf-sup condition within the underlying finite element spaces \cite{matthies2007unified,knobloch2010generalization}, which is crucial for establishing the analysis leading to optimal error convergence.
Compared to residual-based stabilization methods, such as SUPG, LPS provides similar stability advantages without requiring the computation of second-order derivatives, thus making it particularly suitable for time-dependent problems. Furthermore, a distinguishing benefit of the LPS approach is the symmetry of its stabilization term, which is especially advantageous in optimization contexts \cite{braack2009optimal}.

There are two primary variants of LPS: 
the macroelement-based method (covering both the original two-level method \cite{becker2001finite} and its generalization with local projection spaces on overlapping sets \cite{knobloch2010generalization,garg2023local}), and the bubble-enrichment method \cite{matthies2007unified,ganesan2010stabilization}.
We focus on the bubble-enrichment method, which features a more compact stencil than the macroelement-based method and is more efficient in terms of implementation. 

Since our analysis is concerned with the regime $\varepsilon \ll |\bm{\beta}|$, the diffusion term becomes negligible in the asymptotic limit. 
Consequently, we omit the diffusion term and restrict our investigation to the following vector advection problems:

\begin{equation}
    \label{eq:3d h curl}
    \left\{
    \begin{aligned}
        (\nabla \times \bm{u}) \times \bm{\beta}
        + \nabla (\bm{\beta} \cdot \bm{u}) + \gamma \bm{u} &= \bm{f} \quad \text{in } \Omega, \\
        \bm{u} &= \bm{g} \quad \text{on } \Gamma_{-},
    \end{aligned}
    \right.
\end{equation}

\begin{equation}
    \label{eq:3d h div}
    \left\{
    \begin{aligned}
        \bm{\beta} (\nabla \cdot \bm{u}) + \nabla \times (\bm{u} \times \bm{\beta}) + \gamma \bm{u} &= \bm{f} \quad \text{in } \Omega, \\
        \bm{u} &= \bm{g} \quad \text{on } \Gamma_{-}.
    \end{aligned}
    \right.
\end{equation}
Here, $\Omega \subset \mathbb{R}^3$ is a bounded, Lipschitz, polyhedral domain with boundary $\Gamma = \partial \Omega$. We assume that the velocity field $\bm{\beta}(x)$ belongs to $\bm{W}^{1,\infty}(\Omega)$ and that the reaction coefficient $\gamma(x)$ is in $L^\infty(\Omega)$. Furthermore, we consider source terms $\bm{f} \in \bm{L}^2(\Omega)$ and boundary data $\bm{g} \in \bm{L}^2(\Gamma_{-})$. The inflow and outflow parts of the boundary are defined in the usual way:
\begin{align*}
    \Gamma_- &:= \{x \in \Gamma : \bm{\beta}(x) \cdot \bm{n}(x) < 0\}, \\
    \Gamma_+ &:= \{x \in \Gamma : \bm{\beta}(x) \cdot \bm{n}(x) \geq 0\},
\end{align*}
where $\bm{n}(x)$ denotes the unit outward normal vector to $\Gamma$ at point $x \in \Gamma$.

In this paper, we present a unified analysis of the vector advection equations \eqref{eq:3d h curl} and \eqref{eq:3d h div} using a local projection stabilization (LPS) method of arbitrary order. To the best of our knowledge, this represents the first development of an LPS method specifically tailored for vector advection problems. The key innovation is the construction of suitable $\bm{H}(\mathrm{d})$-conforming bubble spaces (with $\mathrm{d} = \mathrm{curl},\mathrm{div}$) along with establishing a corresponding local inf-sup condition. By enriching the finite element spaces with these specialized bubble functions, we show the existence of a modified interpolation operator $j_h$ that satisfies an orthogonality property in addition to the standard approximation properties. 
The availability of this operator enables a rigorous proof of optimal-order convergence for the proposed LPS method, representing the first detailed convergence analysis of an LPS approach for vector advection equations.

The remainder of this paper is organized as follows. In Section \ref{sec:preliminary}, we introduce the necessary preliminary results. Section \ref{sec:method} provides details of the discretization method. 
In Section \ref{sec:bubble}, we construct suitable bubble spaces, allowing us to establish existence of a modified interpolation operator, which is fundamental in our LPS approach. Section \ref{sec:estimate} is devoted to deriving optimal error estimates.
Numerical examples confirming our theoretical findings are presented in Section \ref{sec:numerical}. 
Finally, we conclude the paper with our conclusion in Section \ref{sec:conclusion}.
\section{Preliminaries}\label{sec:preliminary}
We use standard notation for Sobolev spaces, inner products, and the associated norms. Let $D \subset \mathbb{R}^3$ be a bounded domain. For any positive integer $s$ and exponent $p \in [1, \infty]$, we denote by $W^{s,p}(D)$ the Sobolev space equipped with its usual norm $\|\cdot\|_{s,p,D}$ and semi-norm $|\cdot|_{s,p,D}$. In the special case $p = 2$, we write $H^s(D) := W^{s,2}(D)$, with the corresponding norm and semi-norm simplified to $\|\cdot\|_{s,D} := \|\cdot\|_{s,2,D}$ and $|\cdot|_{s,D} := |\cdot|_{s,2,D}$, respectively. The $L^2$-inner product in $L^2(D)$ or $[L^2(D)]^d$ is denoted by $(\cdot, \cdot)_D$, and when the context is clear, we abbreviate it as $(\cdot, \cdot)$.

We denote the $\bm{H}(\mathrm{curl})$ advection operator and its formal adjoint by

 \begin{equation*}
	\mathcal{L}_{\bm{\bm{\beta}}}^{\mathrm{curl}} \bm{u} := \nabla (\bm{\bm{\beta}} \cdot \bm{u}) + (\nabla \times \bm{u}) \times \bm{\bm{\beta}},\quad
	(\mathcal{L}^{\mathrm{curl}}_{\bm{\bm{\beta}}})^{*} \bm{u} :=  - \bm{\bm{\beta}} \nabla\cdot \bm{u} + \nabla \times (\bm{\bm{\beta}}\times \bm{u}).
\end{equation*}
Similarly, the $H(\mathrm{div})$ advection operator and its formal adjoint are defined as
  \begin{equation*}
	\mathcal{L}^{\mathrm{div}}_{{\bm{\beta}}} \bm{u} :=  \bm{\bm{\beta}} \nabla\cdot \bm{u} - \nabla \times (\bm{\bm{\beta}}\times \bm{u}),\quad 
	(\mathcal{L}^{\mathrm{div}}_{\bm{\bm{\beta}}})^{*} \bm{u} := -\nabla (\bm{\bm{\beta}} \cdot \bm{u}) - (\nabla \times \bm{u}) \times \bm{\bm{\beta}}. 
\end{equation*}

For the sake of a unified analysis, we write the advection operator generically as
$\mathcal{L}_{\bm{\beta}} = \mathcal{L}^{\mathrm{curl}}_{\bm{\beta}}  \text{ or }  \mathcal{L}^{\mathrm{div}}_{\bm{\beta}}$,
depending on the context. Using this unified notation, equations~\eqref{eq:3d h curl} and~\eqref{eq:3d h div} can be written in a unified form
\begin{equation}
              \label{eq:3d unified}
						\left\{
						\begin{aligned}
		                \mathcal{L}_{{\bm{\beta}}} \bm{u} +  \gamma
							\bm{u}&= \bm{f}  &&\text{in }\Omega, \\
							\bm{u}&= \bm{g}  &&\text{on } \Gamma_{-}.
							\end{aligned}
						\right.
		\end{equation}
Additionally, we have an integration by parts formula for the advection operator
\begin{equation}
\label{int-by-parts}
(\mathcal{L}_{\bm{\bm{\beta}}}  \bm{u}, \bm{v})_D - (\bm{u}, \mathcal{L}^{*}_{\bm{\bm{\beta}}} \bm{v})_{D} = 
\left\langle\bm{u}, \bm{v}\right\rangle_{\partial D, \bm{\bm{\beta}}},
\end{equation}
where the boundary term is defined by $\left\langle\bm{u}, \bm{v}\right\rangle_{F, \bm{\bm{\beta}}} := \int_{F} (\bm{\bm{\beta}} \cdot \bm{n})(\bm{u} \cdot \bm{v}) dS$. Here, $F$ denotes either a single facet or a union of such facets, typically lying on the boundary of the domain.

Let $\mathcal{T}_h$ be a shape-regular triangulation of the domain $\Omega\subset \mathbb{R}^3$, composed of 3-dimensional simplex elements $K$. The diameter of an element $K$ is denoted by $h_K$, while the mesh parameter $h$ is defined as $h := \max_{K \in \mathcal{T}_h} h_K$. Each element boundary is composed of four triangular facets. To each facet $f$, we assign a unit normal vector $\bm{n}_f$, which either coincides with or opposes the outward unit normal $\bm{n}_{\partial K}$ of the element $K$ containing $f$. For a piecewise smooth vector field $\bm{u}$ defined on $\mathcal{T}_h$, we denote its two  traces on a shared facet $f$ by $\bm{u}^+$ and $\bm{u}^-$, where $\bm{u}^+ := \bm{u}|_{K^+}$ corresponds to the element $K^+$ whose outward normal coincides with $\bm{n}_f$. The jump and average operators are then defined as  
\[
[\bm{u}]_f := \bm{u}^+ - \bm{u}^-, \quad \{\bm{u}\}_f := \frac{1}{2}(\bm{u}^+ + \bm{u}^-).
\]
For boundary facets $f \subset \partial \Omega$, the normal $\bm{n}_f$ is assumed to point outward. We classify the facets into interior facets $\mathcal{F}^\circ$ and boundary facets $\mathcal{F}^\partial$, with the latter further partitioned into inflow facets $\mathcal{F}^\partial_-$ and outflow facets $\mathcal{F}^\partial_+$.

Throughout this paper, $C$ denotes a generic positive
constant that may depend on the shape-regularity of the triangulation but is independent of the mesh size.
 Let $\overline{\bm{\beta}}$ be a piecewise constant approximation of $\bm{\beta}$, satisfying the following approximation properties on each element $K\in\mathcal{T}_h$
\begin{equation}
    \begin{aligned}
        \|\overline{\bm{\beta}}\|_{0,\infty,K} &\leq \|\bm{\beta}\|_{0,\infty,K}, \\
        \|\overline{\bm{\beta}} - \bm{\beta}\|_{0,\infty,K} &\leq C h_K |\bm{\beta}|_{1,\infty,K},
    \end{aligned}
\end{equation}
where $C$ is a constant independent of the mesh size $h$.

For the sake of the well-posedness of equation \eqref{eq:3d unified}, we assume throughout the paper that
\begin{equation}
\label{coer:positive eig}
						\rho^\pm(x) := \lambda_{\min} \left[ 
						\gamma I  \pm  (\frac{\nabla
							{{\bm{\beta}}} + (\nabla {{\bm{\beta}}})^T}{2} - \frac{\nabla \cdot \bm \beta }{2}I)\right] \geq {\rho_0 >
							0}, \quad x\in \Omega,
\end{equation}
where $\rho_0$ is a constant. Here, the sign is positive if $\mathrm{d} = \mathrm{curl}$ and negative if $\mathrm{d} = \mathrm{div}$. This assumption may appear somewhat unusual, but it is the counterpart of the scalar condition $\gamma - \frac{1}{2} \mathrm{div} , \bm{\beta} > \rho_0$ within the framework of Friedrichs symmetric operators \cite{friedrichs1958symmetric}. For a detailed derivation in the case $d = \mathrm{curl}$, see \cite[Section 3]{heumann2013stabilized}.

Let $P_{r}(K)$ be the space of polynomials on $K$ with degree less than or equal to $r$ and $\bm{P_{r}}(K):=[P_{r}(K)]^3$ denote its vector-valued counterpart.   We will frequently use the following trace inequality (see, e.g., \cite{brenner2008mathematical}).
\begin{lemma}[trace inequality]
   Let $f$ be a face of $K \in \mathcal{T}_h$. For $v|_K \in H^1(K)$ the following estimate holds
\begin{equation*}
    \|v\|_{L^2(f)} \leq C \left( h_K^{-1/2} \|v\|_{L^2(K)} + h_K^{1/2} \|\nabla v\|_{L^2(K)} \right),
\end{equation*} 
where C is a constant independent of $h_K$.
\end{lemma}
We will also make use of the standard inverse inequality (see \cite{brenner2008mathematical}).
\begin{lemma}[inverse inequality]
 Let $v \in P_{r}(K)$ for some $r \geq 0$. Then the following estimate holds
\begin{equation*}
   \|\nabla v\|_{L^2(K)} \leq C h_K^{-1} \|v\|_{L^2(K)},
\end{equation*} 
where C is a constant independent of $h_K$.
\end{lemma}

\section{Local projection stabilization (LPS) methods}\label{sec:method}
In this section, we present the proposed LPS methods. Let $\bm{W}_h$ be an $\bm H(\mathrm{d})$ conforming space  with d~= curl or div, satisfying the following approximation property for $r \geq 1$:
\begin{equation}
\label{W_h:i_h_appro}
\begin{split}
\|\bm{u} - i_h \bm{u}\|_{l,K} &\leq C h^{r+1-l} \|\bm{u}\|_{r+1,K},\text{ for } l=0,1, \quad \forall K \in \mathcal{T}_h, \ \forall \bm{u} \in \bm{H}^{r+1}(\Omega),
\end{split}
\end{equation}
where $i_h$ is the canonical interpolation. 
For $\mathrm{d} = \mathrm{curl}$, $\bm{W}_h$ can be taken as the second-kind $r$-th order N\'ed\'elec space, while for $\mathrm{d} = \mathrm{div}$, $\bm{W}_h$ may be chosen as the $r$-th order BDM space (see, e.g., \cite{boffi2013mixed}).
The finite element space $\bm{V}_h$ for approximating the advection problem is defined as
\begin{equation}
\label{def: fem space}
\bm{V}_h:=\bm{W}_h+\bm{B}_h,
\end{equation}
 where $\bm{B}_h := \bm{B}_h^{\mathrm{curl}}$ or $\bm{B}_h^{\mathrm{div}}$ is an $\bm H(\mathrm{d})$ bubble space (with $\mathrm{d} = \mathrm{curl}, \mathrm{div})$ to be defined later.

Define the projection space
\begin{equation}
\label{def: D_h}
    \bm{D}_h  :=\{ \bm{v} \in L^2(\Omega) : \bm{v}|_K \in \bm{P}_{r-1}(K), \quad \forall K \in \mathcal{T}_h \},
\end{equation}
and let $\pi_h$ denote the $L^2$ projection from $\bm{L^2}(\Omega)$ onto $\bm{D}_h$. We introduce the fluctuation operator $\kappa_h := id-\pi_h$, where $id$ is the identity operator. The following approximation property holds
\begin{equation}
\label{D_h appro}
\|\kappa_h\bm{q}\|_{0,K} \leq C h^{l}|\bm{q}|_{l,K}, \quad \forall  0\leq l \leq r, ~\forall \bm{q} \in \bm{H}^r(\mathcal{T}_h).
\end{equation}

We derive the proposed scheme as follows. For a fixed element $K$, we test the variational form \eqref{eq:3d unified} with a function $\bm{v} \in \bm{V}_h$, integrate over $K$, and apply the integration-by-parts identity \eqref{int-by-parts} to obtain
$$
(\gamma \bm{u}, \bm{v})_K + (\bm{u}, \mathcal{L}^{*}_{\bm{\beta}} \bm{v})_K + \langle \bm{u}, \bm{v}\rangle_{\partial K, \bm{\beta}} = (f, \bm{v})_K.
$$
Summing over all elements $K$ gives:
$$
(\gamma \bm{u}, \bm{v})_\Omega + \sum_K (\bm{u}, \mathcal{L}^{*}_{\bm{\beta}} \bm{v})_K + \sum_K \langle \bm{u}, \bm{v}\rangle_{\partial K, \bm{\beta}} = (f, \bm{v})_\Omega.
$$
Rewriting the boundary term as a sum over facets yields
\[
(\gamma \bm{u}, \bm{v})_\Omega + \sum_K (\bm{u}, \mathcal{L}^{*}_{\bm{\beta}} \bm{v})_K + \sum_{f \in \mathcal{F}^\circ} \left( \langle \bm{u}^+, \bm{v}^+ \rangle_{f, \bm{\beta}} - \langle \bm{u}^-, \bm{v}^- \rangle_{f, \bm{\beta}} \right) + \sum_{f \in \mathcal{F}^\partial} \langle \bm{u}, \bm{v} \rangle_{f, \bm{\beta}} = (f, \bm{v})_\Omega,
\]
where, in the evaluation of the contribution over each facet $f$, the preassigned unit normal vector $\bm{n}_f$ is used. That is, $\langle \bm{u}, \bm{v} \rangle_{f, \bm{\beta}} = \int_f (\bm{\beta} \cdot \bm{n}_f)(\bm{u} \cdot \bm{v})\, dS$.

Next, we use the identity
\[
\langle \bm{u}^+, \bm{v}^+ \rangle_{f, \bm{\beta}} - \langle \bm{u}^-, \bm{v}^- \rangle_{f, \bm{\beta}} 
= \langle [\bm{u}]_f, \{\bm{v}\}_f \rangle_{f, \bm{\beta}} + \langle \{\bm{u}\}_f, [\bm{v}]_f \rangle_{f, \bm{\beta}},
\]
and observe that for smooth solutions $\bm{u}$ of the advection problem~\eqref{eq:3d unified}, the jump term vanishes, so that
\[
\langle \bm{u}^+, \bm{v}^+ \rangle_{f, \bm{\beta}} - \langle \bm{u}^-, \bm{v}^- \rangle_{f, \bm{\beta}} 
= \langle \{\bm{u}\}_f, [\bm{v}]_f \rangle_{f, \bm{\beta}}.
\]
This simplification holds since $\bm{u}$ is only non-smooth across characteristic faces where $\bm{\beta} \cdot \bm{n}_f = 0$, and for such faces the term $\langle \cdot, \cdot \rangle_{f, \bm{\beta}}$ vanishes.

The local projection discretization of \eqref{eq:3d unified} reads as: find $\bm{u}_h \in \bm{V}_h$ such that
\begin{equation}
\label{scheme}
a(\bm{u}_h, \bm{v}_h)+ S_h(\bm{u}_h, \bm{v}_h) = l(\bm{v}_h), \quad \forall \bm{v}_h \in \bm{V}_h,
\end{equation}
where the bilinear and linear forms are defined by
\[
\begin{aligned}
  a(\bm{u}, \bm{v}) :=& (\gamma \bm{u}, \bm{v})_\Omega + \sum_{K} (\bm{u}, \mathcal{L}^{*}_{{\bm{\beta}}} \bm{v})_K + \sum_{f \in \mathcal{F}^{\partial}_+} \left\langle\bm{u}, \bm{v}\right\rangle_{f, {\bm{\beta}}}
+ \sum_{f \in \mathcal{F}^{\circ}} \left\langle \{ \bm{u} \}_f , [\bm{v}]_f \right\rangle_{f, {\bm{\beta}}}, \\
l(\bm{v}) :=& (\bm{f}, \bm{v})_\Omega - \sum_{f \in \mathcal{F}_-^{\partial}} \left\langle\bm{g}, \bm{v}\right\rangle_{f, {\bm{\beta}}}.
\end{aligned}
\]
Here, the stabilization term is defined as
\begin{equation}
\begin{aligned}
    S_h(\bm{u}, \bm{v}) :=& S_h^1(\bm{u}, \bm{v})+S_h^2(\bm{u}, \bm{v})\\
:=&\sum_{f \in \mathcal{F}^{\circ}}\left\langle c_f [\bm{u}]_f , [\bm{v}]_f \right\rangle_{f, \bm{\bm{\beta}}}
+\sum_{K \in \mathcal{T}_h} h_K  (\kappa_h(\mathcal{L}^{*}_{\overline{\bm{\beta}}} \bm{u}), 
\kappa_h(\mathcal{L}^{*}_{\overline{\bm{\beta}}} \bm{v}))_K,
\end{aligned}
\end{equation}
where $c_f$ is a facewise constant of $\mathcal{O}(1)$ such that
\begin{equation}
\label{cf positive}
c_f \bm{\beta}\cdot \bm{n}_f > 0.  
\end{equation}
Noting that the approximation $\overline{\bm{\beta}}$ rather than $\bm{\beta}$ is adopted in $S_h^2(\cdot,\cdot)$.

We also define the following energy norm for LPS methods as 
\begin{equation}
\label{ene norm}
\begin{aligned}
  \|\bm{u}\|_h^2 := &\|\bm{u}\|^2_{L^2(\Omega)}  + \sum_{f \in \mathcal{F}^{\partial}_+} \|\bm{u}\|^2_{f, \frac{1}{2} {\bm{\beta}}} + \sum_{f \in \mathcal{F}^{\partial}_{-}} \|\bm{u}\|^2_{f, -\frac{1}{2} {\bm{\beta}}} +\sum_{f \in \mathcal{F}^{\circ}} \|[\bm{u}]_f\|^2_{f, c_f {\bm{\beta}}} + \sum_{K \in \mathcal{T}_h} h_K \|\kappa_h(\mathcal{L}^{*}_{\overline{\bm{\beta}}} \bm{u})\|^2_{0,K}.   
\end{aligned}
\end{equation}
After introducing the abstract framework for local projection stabilization, we will present the construction of the space $\bm{V}_h$ along with the design of modified interpolation operators, which form the foundation of the proposed method.

\section{Bubble space and modified interpolation operator}\label{sec:bubble}

In this section, we construct the bubble spaces $\bm{B}_h^{\mathrm{curl}}$ and $\bm{B}_h^{\mathrm{div}}$  which are essential for ensuring the existence of a modified interpolation operator that satisfies both an orthogonality condition and the standard approximation property.

 Let ${\lambda}_0$,  ${\lambda}_1$, ${\lambda}_2$, ${\lambda}_3$ denote the barycentric coordinates associated with the four vertices of a tetrahedral element $K\in \mathcal{T}_h$. We define  the $\bm H(\mathrm{curl})$ conforming bubble space on each element $K$ as
\begin{equation}
    \label{curl bubble}
    \bm{B}_h^{\mathrm{curl}}(K):=P_{r-1}(K)\bm{b}_1^{\rm curl}+P_{r-1}(K) \bm{b}_2^{\rm curl}+P_{r-1}(K) \bm{b}_3^{\rm curl}, \quad \forall K \in \mathcal{T}_h,
\end{equation}
where the bubble functions are given by ${\bm{b}}_1^{\rm curl} := {\lambda}_2 {\lambda}_3 {\lambda}_0 {\bm{n}}_1$, 
${\bm{b}}_2^{\rm curl} := {\lambda}_3 {\lambda}_0{\lambda}_1 {\bm{n}}_2$,
${\bm{b}}_3^{\rm curl} := {\lambda}_0 {\lambda_1} {\lambda_2} {\bm{n}}_3$, and
$\bm{n_1}$, $\bm{n_2}$, $\bm{n_3}$ are the unit outward normal vectors to arbitrary three faces of $K$. By definition, it is clear that vector functions in $\bm{B}_h^{\mathrm{curl}}(K)$ vanish in their tangential components on the boundary $\partial K$.

Similarly, the $\bm H(\mathrm{div})$ conforming bubble space is defined as
\begin{equation}
    \label{div bubble}
    \bm{B}_h^{\mathrm{div}}(K):=P_{r-1}(K)\bm{b}_1^{\rm div}+P_{r-1}(K) \bm{b}_2^{\rm div}+P_{r-1}(K) \bm{b}_3^{\rm div}, \quad \forall K \in \mathcal{T}_h,
\end{equation}
with bubble functions ${\bm{b}}_1^{\rm div} := {\lambda}_0 {\lambda}_1 \bm{t}_{01}$, 
${\bm{b}}_2^{\rm div} := {\lambda}_0 {\lambda}_2 \bm{t}_{02}$,
${\bm{b}}_3^{\rm div} := {\lambda}_0 {\lambda_3}  \bm{t}_{03}$,
where $\bm{t}_{01}$, $\bm{t}_{02}$ and $\bm{t}_{03}$ are the unit tangent vectors along arbitrary three edges of $K$ that share the common vertex. Similarly, it is clear that vector functions in $\bm{B}_h^{\mathrm{div}}(K)$ have zero normal components on the boundary $\partial K$. See Figure~\ref{fig: normal and tangent} for an illustration of the geometric configuration of the normal and tangent vectors used in the construction.


\usetikzlibrary{3d, arrows.meta, calc}

\begin{figure}[!htbp]
    \centering
    \begin{tikzpicture}[scale=3, >=Stealth]

\begin{scope}[xshift=-2cm]
    
    \coordinate (0) at (0,0,0);
    \coordinate (1) at (1,0,0);
    \coordinate (2) at (0.5,0.866,0);
    \coordinate (3) at (1,0.289,1.616);
    
    \draw[dashed, thick] (0) -- (1); 
    \draw[thick] (0) -- (2);
    \draw[thick] (1) -- (2);
    \draw[thick] (0) -- (3);
    \draw[thick] (1) -- (3);
    \draw[thick] (2) -- (3);
    
    \node[left] at (0) {0};
    \node[right] at (1) {1};
    \node[above] at (2) {2};
    \node[below left] at (3) {3};
    
     \draw[->, black, thick] (0.5,0.385,0.272) -- ++(-0.5,-0.289,-0.816) node[below] {$\bm{n}_1$};
    \coordinate (n1start) at (0.7,0,0.5);
    \coordinate (n1end) at ($(n1start)+(0,-0.3,0)$);
    \draw[dashed, thick] (n1start) -- ($(n1start)!0.25!(n1end)$); 
    \draw[->, solid, thick] ($(n1start)!0.25!(n1end)$) -- (n1end) node[right] {$\bm{n}_2$}; 

    \coordinate (n1start) at (0.6,0.4,0);
    \coordinate (n1end) at ($(n1start)+(0,0,-0.5)$);
    \draw[dashed, thick] (n1start) -- ($(n1start)!0.6!(n1end)$); 
    \draw[->, solid, thick] ($(n1start)!0.6!(n1end)$) -- (n1end) node[right] {$\bm{n}_3$}; 
\end{scope}

\begin{scope}[xshift=-0.3cm]
     \coordinate (0) at (0,0,0);
    \coordinate (1) at (1,0,0);
    \coordinate (2) at (0.5,0.866,0);
    \coordinate (3) at (1.3,0.289,1.616);
    
    \draw[dashed, thick] (0) -- (1); 
    \draw[thick] (0) -- (2);
    \draw[thick] (1) -- (2);
    \draw[thick] (0) -- (3);
    \draw[thick] (1) -- (3);
    \draw[thick] (2) -- (3);
    
    \node[left] at (0) {0};
    \node[right] at (1) {1};
    \node[above] at (2) {2};
    \node[below left] at (3) {3};
    
    \draw[->,dashed, thick] (0) -- ($0.6*($(1)-(0)$)$) node[pos=0.6, above] {$\bm{t}_{01}$};
    \draw[->, thick] (0) -- ($0.5*($(2)-(0)$)$) node[pos=0.6, left] {$\bm{t}_{02}$};
    \draw[->, thick] (0) -- ($0.8*($(3)-(0)$)$) node[pos=0.8, left] {$\bm{t}_{03}$};
\end{scope}
\end{tikzpicture}
   \caption{\centering Normal and tangent vectors associated with the $H(\mathrm{d})$ bubble functions.}
    \label{fig: normal and tangent}
\end{figure}
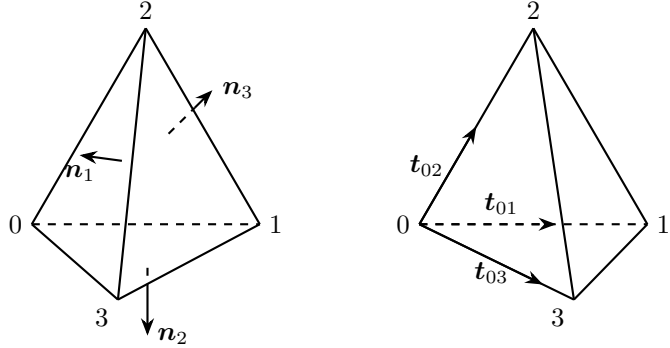

The following pair of lemmas establishes local inf-sup conditions, which are a key ingredient in our analysis.
\begin{lemma}[curl local inf-sup]
\label{curl:local inf sup}
Let $\bm{B}_h^{\mathrm{curl}}(K)$ and $\bm{D}_h(K) := \bm{D}_h|_K$ be given as in \eqref{curl bubble} and \eqref{def: D_h}, respectively, for a fixed $r \geq 1$. Then, the local inf-sup condition holds
\begin{equation}
\label{main ingre}
\inf_{\bm{q}_h \in \bm{D}_h(K)} \sup_{\bm{v}_h \in \bm{B}_h^{\mathrm{curl}}(K)} \frac{({\bm{v}_h}, \bm{q}_h)_K}{\|\bm{v}_h\|_{0,K} \|\bm{q}_h\|_{0,K}} \ge c > 0,
\end{equation}  
where $c$ is a constant independent of $h_K$.
\end{lemma}

\begin{proof}
For simplicity, we omit the superscript $\mathrm{curl}$ and write $\bm{B}_h(K) := \bm{B}_h^{\mathrm{curl}}(K)$. For the reference element $\hat{K}$, define $\bm{\hat{B}}(\hat{K})$ as a counterpart of $\bm{B}_h(K)$ by
\[\bm{\hat{B}}(\hat{K})=P_{r-1}(\hat{K})\hat{\bm{b}}_1^{\rm curl}+P_{r-1}(\hat{K})\hat{\bm{b}}_2^{\rm curl}+P_{r-1}(\hat{K})\hat{\bm{b}}_3^{\rm curl},
\] 
where  $$\hat{\bm{b}}_1^{\rm curl} = \hat{\lambda}_2 \hat{\lambda}_3 \hat{\lambda}_0 \hat{\bm{n}}_1,\quad 
\hat{\bm{b}}_2^{\rm curl} = \hat{\lambda}_3 \hat{\lambda}_0\hat{\lambda}_1 \hat{\bm{n}}_2,\quad 
\hat{\bm{b}}_3^{\rm curl} = \hat{\lambda}_0 \hat{\lambda_1} \hat{\lambda_2} \hat{\bm{n}}_3,$$
and $\hat{\bm{n}}_i$ are unit outward normals to three faces of $\hat{K}$.

Let $F_K(\hat{x})=B\hat{x}+b$ be the affine mapping from $\hat{K}$  to $K$. Denote the Piola transformations (see, e.g. \cite{boffi2013mixed})
\begin{equation*}
\mathcal{G}\hat{\bm{q}}_h :=|\det(B)|^{-1}B\hat{\bm{q}}_h \circ F_K^{-1},
\end{equation*}
\begin{equation*}
\mathcal{H}\hat{\bm{v}}_h :=B^{-T}\hat{\bm{v}}_h \circ F_K^{-1}.
\end{equation*}
Taking $\hat{\bm{q}}_h $ and $\hat{\bm{v}}_h$ such that $\bm{q}_h=\mathcal{G}\hat{\bm{q}}_h $,  $\bm{v}_h=\mathcal{H}\hat{\bm{v}}_h$, then we have  
\begin{equation*}
(\bm{v}_h, \bm{q}_h)_K=(\hat{\bm{v}}_h,\hat{\bm{q}}_h )_{\hat{K}}.
\end{equation*}
Standard scaling arguments yield (see, e.g., \cite{boffi2013mixed})
\begin{equation*}
\|\bm{q}_h\|_{0,K} \leq |\det(B)|^{-1/2}\|B \| \|\hat{\bm{q}}_h \|_{0,K},
\quad
 \|\bm{v}_h\|_{0,K} \leq |\det(B)|^{1/2}\|B^{-1} \| \|\hat{\bm{v}}_h\|_{0,K}.
\end{equation*}
Shape-regularity of the mesh then gives
\begin{equation}
\inf_{\bm{q}_h \in \bm{D}_h(K)} \sup_{\bm{v}_h \in \bm{B}_h(K)} \frac{(\bm{v}_h, \bm{q}_h)_K}{\|\bm{v}_h\|_{0,K} \|\bm{q}_h\|_{0,K}} 
\ge C
\inf_{\hat{\bm{q}}_h \in \bm{P_{r-1}}(\hat{K})} \sup_{\hat{\bm{v}}_h \in \bm{\hat{B}}(\hat{K})} \frac{(\hat{\bm{v}}_h,\hat{\bm{q}}_h)_{\hat{K}}}{\|\hat{\bm{v}}_h\|_{0,{\hat{K}}} \|\hat{\bm{q}}_h\|_{0,\hat{K}}} .
\label{eq:infsuptest}
\end{equation}
Given $\hat{\bm{q}}_h \in \bm{P_{r-1}}(\hat{K})$, we define
\begin{equation*}
\hat{\bm{v}}_h := \sum_{i=1}^3 (\hat{\bm{n}}_i \cdot \hat{\bm{q}}_h) \hat{\bm{b}}_i^{\rm curl}\in \bm{\hat{B}}(\hat{K}).
\end{equation*}
Then we have
\begin{equation}
\begin{aligned}
(\hat{\bm{v}}_h, \hat{\bm{q}}_h)_{\hat{K}} 
&= \int_{\hat{K}} \hat{\lambda}_2 \hat{\lambda}_3 \hat{\lambda}_0 |\hat{\bm{q}}_h \cdot \hat{\bm{n}}_1|^2 + \hat{\lambda}_3 \hat{\lambda}_0 \hat{\lambda}_1|\hat{\bm{q}}_h \cdot \hat{\bm{n}}_2|^2
+\hat{\lambda}_0 \hat{\lambda}_1 \hat{\lambda}_2|\hat{\bm{q}}_h \cdot \hat{\bm{n}}_3|^2 d\hat{x} \\
&\ge \int_{\hat{K}} ( |\hat{\bm{q}}_h \cdot \hat{\bm{n}}_1|^2 + |\hat{\bm{q}}_h \cdot \hat{\bm{n}}_2|^2 + |\hat{\bm{q}}_h \cdot \hat{\bm{n}}_3|^2 
) \hat{\lambda}_0 \hat{\lambda}_1 \hat{\lambda}_2 \hat{\lambda}_3  d\hat{x} \\
&\ge C \int_{\hat{K}} |\hat{\bm{q}}_h|^2 d\hat{x}, \label{inner pro}
\end{aligned}
\end{equation}
using equivalence of norms over finite-dimensional spaces and positivity of barycentric coordinate functions. Moreover, since $\|\hat{\bm{b}}_i\|_{0,\infty} \le 1$, we have
\begin{equation}
\left\| \hat{\bm{v}}_h \right\|_{0,\hat{K}} = \sum_{i=1}^3 \left\| (\hat{\bm{n}}_i \cdot \hat{\bm{q}}_h) \hat{\bm{b}}_i^{\rm curl} \right\|_{0,\hat{K}} \leq C \left\| \hat{\bm{q}}_h \right\|_{0,\hat{K}}.
\label{norm}
\end{equation}
Combining \eqref{eq:infsuptest}, \eqref{inner pro}, and \eqref{norm} gives the desired result.
\end{proof}

\begin{lemma}[div local inf-sup]
\label{div:local inf sup}
Let $\bm{B}_h^{\mathrm{div}}(K)$ and $\bm{D}_h(K) := \bm{D}_h|_K$ be given as in \eqref{div bubble} and \eqref{def: D_h}, respectively, for a fixed $r \geq 1$. Then the local inf-sup condition holds
\begin{equation}
\label{main ingre div}
\inf_{\bm{q}_h \in \bm{D}_h(K)} \sup_{\bm{\bm{v}_h} \in \bm{B}_h^{\mathrm{div}}(K)} \frac{(\bm{\bm{v}_h}, \bm{q}_h)_K}{\|\bm{v}_h\|_{0,K} \|\bm{q}_h\|_{0,K}} \ge c > 0,
\end{equation}  
with $c$ independent of  of $h_K$.
\end{lemma}

\begin{proof}
The proof follows the same structure as Lemma~\ref{curl:local inf sup}. We omit the superscript div and denote $\bm{B}_h(K) := \bm{B}_h^{\mathrm{div}}(K)$. Define the reference bubble space $\bm{\hat{B}}(\hat{K})$ analogously, 
\[\bm{\hat{B}}(\hat{K})=P_{r-1}(\hat{K})\hat{\bm{b}}_1^{\rm div}+P_{r-1}(\hat{K})\hat{\bm{b}}_2^{\rm div}+P_{r-1}(\hat{K})\hat{\bm{b}}_3^{\rm div}.
\] 
Unlike the curl case, we apply the Piola transformations by choosing $\hat{\bm{q}}_h$ and $\hat{\bm{v}}_h$ such that $\bm{q}_h = \mathcal{H} \hat{\bm{q}}_h$ and $\bm{v}_h = \mathcal{G} \hat{\bm{v}}_h$, with the property
\begin{equation*}
(\bm{v}_h, \bm{q}_h)_K=(\hat{\bm{v}}_h,\hat{\bm{q}}_h )_{\hat{K}}, 
\end{equation*}
and by the same reasoning as before,
\begin{equation}
\inf_{\bm{q}_h \in \bm{D}_h(K)} \sup_{\bm{v}_h \in \bm{B}_h(K)} \frac{(\bm{v}_h, \bm{q}_h)_K}{\|\bm{v}_h\|_{0,K} \|\bm{q}_h\|_{0,K}} 
\ge C
\inf_{\hat{\bm{q}}_h \in \bm{P_{r-1}}(\hat{K})} \sup_{\hat{\bm{v}}_h \in \bm{\hat{B}}(\hat{K})} \frac{(\hat{\bm{v}}_h,\hat{\bm{q}}_h)_{\hat{K}}}{\|\hat{\bm{v}}_h\|_{0,{\hat{K}}} \|\hat{\bm{q}}_h\|_{0,\hat{K}}} .
\label{eq:infsuptest div}
\end{equation}
Given $\hat{\bm{q}}_h \in \bm{P_{r-1}}(\hat{K})$, let 
\begin{equation*}
\hat{\bm{v}}_h := \sum_{i=1}^3 (\hat{\bm{t}}_{0i} \cdot \hat{\bm{q}}_h) \hat{\bm{b}}_i^{\rm div}\in \bm{\hat{B}}(\hat{K}), 
\end{equation*}
then we derive
\begin{equation}
\label{inner-pro div}
\begin{aligned}
(\hat{\bm{v}}_h, \hat{\bm{q}}_h)_{\hat{K}} &= \int_{\hat{K}} \hat{\lambda}_0 \hat{\lambda}_1  |\hat{\bm{q}}_h \cdot \hat{\bm{t}}_{01}|^2 + \hat{\lambda}_0 \hat{\lambda}_2|\hat{\bm{q}}_h \cdot \hat{\bm{t}}_{02}|^2
+\hat{\lambda}_0 \hat{\lambda}_3 |\hat{\bm{q}}_h \cdot \hat{\bm{t}}_{03}|^2 d\hat{x}    \\
&\ge \int_{\hat{K}} ( |\hat{\bm{q}}_h \cdot \hat{\bm{t}}_{01}|^2 + |\hat{\bm{q}}_h \cdot \hat{\bm{t}}_{02}|^2 + |\hat{\bm{q}}_h \cdot \hat{\bm{t}}_{03}|^2) \hat{\lambda}_0 \hat{\lambda}_1 \hat{\lambda}_2 \hat{\lambda}_3  d\hat{x} \\
&\ge C \int_{\hat{K}} |\hat{\bm{q}}_h|^2 d\hat{x},
\end{aligned}
\end{equation}
and similarly,
\begin{equation}
\left\| \hat{\bm{v}}_h \right\|_{0,\hat{K}} = \sum_{i=1}^3 \left\| (\hat{\bm{t}}_{0i} \cdot \hat{\bm{q}}_h) \hat{\bm{b}}_i^{\rm div} \right\|_{0,\hat{K}} \leq C \left\| \hat{\bm{q}}_h \right\|_{0,\hat{K}}.
\label{norm div}
\end{equation}
Substituting \eqref{inner-pro div} and \eqref{norm div} into \eqref{eq:infsuptest div} gives the desired result.
\end{proof}

	The preceding two lemmas treat the cases $\mathrm{d} = \mathrm{curl}$ and $\mathrm{d} = \mathrm{div}$ separately. Their arguments fundamentally rely on three key components: the preservation of inner products under the Piola transformation, the structure of the $\bm{H}(\mathrm{d})$-conforming bubble space, and its compatibility with the projection space. Collectively, these lemmas provide the essential foundation for the proof of Theorem~\ref{thm: existence of j_h}, as detailed below.

\begin{remark}
Similar results can also be established for the two-level method \cite{matthies2007unified} and the so-called generalized local projection method \cite{knobloch2010generalization}, both of which do not require bubble enrichment. However, since these approaches are fundamentally based on macroelement local projections, they naturally lead to larger stencils and increased complexity in data structure management. A detailed investigation of these alternative methods will be pursued in future work.
\end{remark}

We now follow the approach used in \cite[Lemma 1]{becker2001finite} and \cite[Theorem 2.2]{matthies2007unified} for the $H(\mathrm{grad})$ case, and extend it to the $\bm H(\mathrm{curl})$ and $\bm H(\mathrm{div})$ settings. Our goal is to establish the existence of a modified interpolation operator $j_h$ that satisfies both an orthogonality property and the standard approximation property. This operator plays a critical role in the error analysis of the proposed method.

\begin{theorem}[modified interpolation operator]
\label{thm: existence of j_h}
Let $\bm{V}_h(K)$ and $\bm{D}_h(K)$ be defined as in \eqref{def: fem space} and \eqref{def: D_h}, then there exists a modified interpolation operator $j_h : \bm{H}^{1+\epsilon}(\Omega) \to \bm{V}_h$ such that the following orthogonality and approximation properties hold:
\begin{equation}
\label{eq:orth}
    (\bm{v} - j_h \bm{\bm{v}}, \bm{q}_h) = 0,  
\quad \forall \bm{q}_h \in \bm{D}_h,  \forall \bm{v} \in \bm{H}^{1+\epsilon}(\Omega),
\end{equation}
and
\begin{equation}
\begin{aligned}
\|\bm{v} - j_h \bm{v}\|_{0,K} +h_K |\bm{\bm{v}} - j_h \bm{v}|_{1,K} \leq C h_K^l \|\bm{v}\|_{l,K},
\quad \forall \bm{v} \in \bm{H}^{1+\epsilon}(\Omega) \cap \bm{H}^l(\Omega), 1 \leq l \leq r + 1.
\end{aligned}
\label{eq:appro}
\end{equation}
Here, $\epsilon > 0$ is an arbitrarily small constant.
\end{theorem} 

\begin{proof}
It follows from Lemma~\ref{curl:local inf sup} and \ref{div:local inf sup} that $\bm{B}_h(K)$ and $\bm{D}_h(K)$ satisfy a local inf-sup condition  
\begin{equation}
\label{eq:local inf sup}
\inf_{\bm{q}_h \in \bm{D}_h(K)} \sup_{\bm{\bm{v}_h} \in \bm{B}_h(K)} \frac{(\bm{\bm{v}_h}, \bm{q}_h)_K}{\|\bm{\bm{v}_h}\|_{0,K} \|\bm{q}_h\|_{0,K}} \ge c > 0,
\end{equation}
with $c$ independent of $h$. This implies that the operator
$A_h:\bm{B}_h(K) \to (\bm{D}_h(K))^\prime$ defined by
$$\langle A_h \bm{v}_h, \bm{q}_h \rangle := (\bm{v}_h,\bm{q}_h),$$
is an isomorphism from $N(A_h)^\perp$ onto $(\bm{D}_h(K))^\prime$.
Therefore, for any $\bm{v} \in \bm{H}^{1+\epsilon}(\Omega)$, there exists a unique function $m_h(\bm{v}) \in N(A_h)^\perp$ such that
\begin{equation}
	\begin{aligned}
\langle A_h m_h(\bm{v}),\bm{q}_h\rangle &:=(m_h(\bm{v}),\bm{q}_h)_K=(\bm{v}-i_h\bm{v},\bm{q}_h)_K,\\
\label{eq:bdd}
\|m_h(\bm{v})\|_{0,K} & \leq {1\over{c}} \|\bm{v} - i_h \bm{v}\|_{0,K},
	\end{aligned}
\end{equation} 
where $i_h$ is the canonical interpolation from $\bm{H}^{1+\epsilon}(\Omega)$ onto $\bm{W}_h$, introduced in \eqref{W_h:i_h_appro}.

We now define the modified interpolation operator $j_h : \bm{H}^{1+\epsilon}(\Omega) \to \bm{V}_h$ based on  $i_h$:
 $$j_h\bm{v}|_K:=i_h\bm{v}|_K+ m_h(\bm{v}).$$ 
By construction, $j_h \bm{v} \in \bm{V}_h $.  Using the triangle inequality and \eqref{eq:bdd}, we have
\begin{equation}
\begin{aligned}
\|\bm{v} - j_h \bm{v}\|_{0,K}
&\leq \|\bm{v} - i_h \bm{v}\|_{0,K} + \| m_h(\bm{v})\|_{0,K}  \\
&\leq \left(1 + \frac{1}{c}\right) \|\bm{v} - i_h \bm{v}\|_{0,K}.  \label{eq:L2 esti}
\end{aligned}
\end{equation}
Applying the inverse inequality to $m_h(\bm{v}) \in \bm{B}_h(K)$, we obtain
\begin{equation}
  \begin{aligned}
|\bm{v} - j_h \bm{v}|_{1,K} &\leq |\bm{v} - i_h \bm{v}|_{1,K} + |m_h(\bm{v})|_{1,K} \\
&\leq |\bm{v} - i_h \bm{v}|_{1,K} + h^{-1} \|m_h(\bm{v})\|_{0,K}   \\
&\leq |\bm{v} - i_h \bm{v}|_{1,K} + h^{-1} \frac{1}{c} \|\bm{v} - i_h \bm{v}\|_{0,K}.\label{eq:H1 esti}
\end{aligned}  
\end{equation}
Combining \eqref{eq:L2 esti}, \eqref{eq:H1 esti}, and the approximation property \eqref{W_h:i_h_appro} of $i_h$, we conclude the estimate \eqref{eq:appro}. The orthogonality condition \eqref{eq:orth} follows directly from the definition of $m_h(\bm{v})$, which ensures that
$$(\bm{v} - j_h \bm{v}, \bm{q}_h) = (\bm{v} - i_h \bm{v} - m_h(\bm{v}), \bm{q}_h) = 0, \quad \forall \bm{q}_h \in \bm{D}_h,$$thus completing the proof.
\end{proof}

\section{Convergence analysis}\label{sec:estimate}
In this section, we analyze the error of the local projection stabilization scheme~\eqref{scheme} with respect to the energy norm $\|\cdot\|_h$ defined in \eqref{ene norm}.
\subsection{Stability}\label{subsec2}
We begin by proving the coercivity of the discrete bilinear form.
\begin{lemma}[coercivity] \label{lm:coercivity} Assume that condition \eqref{coer:positive eig} holds, and that the stabilization coefficient in $S_h^1(\cdot,\cdot)$ satisfies $c_f\, \bm{\beta} \cdot \bm{n}_f > 0$.  
Then the bilinear form is coercive in the norm $\|\cdot\|_h$, i.e.,
\[
a(\bm{u}_h, \bm{u}_h) + S(\bm{u}_h, \bm{u}_h) \geq \min\left\{\frac{1}{2}\rho_0, 1\right\} \|\bm{u}_h\|^2_h,\quad\forall \bm{u}_h \in \bm{V}_h.
\]  
 \end{lemma}

\begin{proof}
A direct calculation yields the identity
\begin{equation}
  \begin{aligned}
\mathcal{L}_{\bm{\bm{\beta}}} \bm{u}_h + \mathcal{L}^{*}_{\bm{\beta}} \bm{u}_h 
&=\pm[ (D{\bm{\beta}}) \bm{u}_h+ (D{\bm{\beta}})^T \bm{u}_h - (\text{div}\bm{\beta}) \bm{u}_h],\label{eq:2.10}
\end{aligned}  
\end{equation}
where the sign is positive if $\mathrm{d} = \mathrm{curl}$ and negative if $\mathrm{d} = \mathrm{div}$.
It follows form the integration by parts formula \eqref{int-by-parts} that
\begin{equation}
   \begin{aligned}
&a(\bm{u}_h, \bm{u}_h) + S(\bm{u}_h, \bm{u}_h) \\
=~& (\gamma \bm{u}_h, \bm{u}_h)_\Omega + \sum_{K} (\bm{u}_h, \mathcal{L}^*_{\bm\beta} \bm{u}_h)_K + \sum_{f \in \mathcal{F}^{\circ}} \langle \{\bm{u}_h\}, [\bm{u}_h]\rangle_{f,\bm{\beta}}+ \sum_{f \in \mathcal{F}_+^{\partial}} \langle \bm{u}_h, \bm{u}_h\rangle_{f, \bm{\beta}}  + S(\bm{u}_h, \bm{u}_h) \\
=~&(\gamma \bm{u}_h, \bm{u}_h)_\Omega + \frac{1}{2} \sum_{K} (\bm{u}_h, \mathcal{L}_{\bm{\bm{\beta}}} \bm{u}_h + \mathcal{L}^{*}_{\bm{\bm{\beta}}} \bm{u}_h)_K \\
&\quad + \frac{1}{2} \sum_{f \in \mathcal{F}^{\partial}_+} \langle \bm{u}_h, \bm{u}_h\rangle_{f, \bm{\beta}} - \frac{1}{2} \sum_{f \in \mathcal{F}^{\partial}_-} \langle \bm{u}_h, \bm{u}_h\rangle_{f, \bm{\beta}} + S(\bm{u}_h, \bm{u}_h) \\
\geq~& \min\left\{\frac{1}{2} \rho_0, 1\right\} \|\bm{u}_h\|^2_h, \quad \forall \bm{u}_h \in \bm{V}_h.
\end{aligned} 
\end{equation}
This completes the proof. 
\end{proof}

\subsection{Consistency error}
Since the scheme~\eqref{scheme} is non-consistent, quantifying the consistency error is essential for the error analysis.
\begin{lemma}[consistency term]                  
    Let $\bm{u}$ be the solution of \eqref{eq:3d unified} and $\bm{u}_h$ be the solution of
    \eqref{scheme}. Then we have
    \begin{equation}
       \label{eq: consistency}
        a(\bm{u} - \bm{u}_h, \bm{v}_h) + S_h^1(\bm{u} - \bm{u}_h, \bm{v}_h) = S_h^2(\bm{u}_h, \bm{v}_h),\quad\forall \bm{v}_h \in \bm{V}_h.
    \end{equation}
\end{lemma}
\begin{proof}
Based on the formulation derived in Section~\ref{sec:method}, and noting that $S_h^1(\bm{u}, \bm{v}_h) = 0$, we obtain
\begin{equation}
\label{111}
a(\bm{u}, \bm{v}_h) + S_h^1(\bm{u}, \bm{v}_h) = l(\bm{v}_h).
\end{equation}
Subtracting \eqref{111} from the numerical scheme \eqref{scheme} yields the consistency relation \eqref{eq: consistency}.
\end{proof}
The following lemma provides an estimate for the consistency error.
\begin{lemma}[consistency error]
    Assume 
    $\bm{u} \in \bm{H}^{r+1}(K)$ for all $K\in\mathcal{T}_h$.  
    Then
\begin{equation}
    \label{eq: consistency err}
     S_h^2(\bm{u}, \bm{v}_h) \leq C \| \bm{v}_h \|_h (\sum_{K \in \mathcal{T}_h} h_K^{2r+1} \|\bm{\beta} \|_{0,\infty,K}^2 \| \bm{u} \|_{r+1,K}^2)^{\frac{1}{2}},\quad\forall \bm{v}_h \in \bm{V}_h.
\end{equation}
\end{lemma}

\begin{proof}
	Apply the approximation property of $\kappa_h$, we have 
 \begin{equation}
    \begin{aligned}
S_h^2(\bm{u}, \bm{u}) &= \sum_{K \in \mathcal{T}_h} h_K \|\kappa_h(\mathcal{L}^{*}_{\overline{\bm{\beta}}} \bm{u})\|^2_{0,K} 
\leq C \sum_{K \in \mathcal{T}_h}h_K h_K^{2r} 
\|\mathcal{L}^{*}_{\overline{\bm{\beta}}} \bm{u}\|^{2}_{r,K}  \\
&\leq C \sum_{K \in \mathcal{T}_h} h_K^{2r+1} \| \overline{\bm{\beta}} \|_{0,\infty,K}^2 \| \bm{u} \|_{r+1,K}^2. \\
\end{aligned}   
 \end{equation}
Therefore, it follows from Cauchy-Schwartz inequality that
\begin{equation}
    \begin{aligned}
S_h^2(\bm{u}, \bm{v}_h) &\leq S_h^2(\bm{u}, \bm{u})^{1/2}S_h^2(\bm{v}_h, \bm{v}_h)^{1/2}\leq S_h^2(\bm{u}, \bm{u})^{1/2}\| \bm{v}_h \|_h  \\
&\leq C \| \bm{v}_h \|_h (\sum_{K \in \mathcal{T}_h} h_K^{2r+1} \| \overline{\bm{\beta}} \|_{0,\infty,K}^2 \| \bm{u} \|_{r+1,K}^2)^{\frac{1}{2}}\\
&\leq C \| \bm{v}_h \|_h (\sum_{K \in \mathcal{T}_h} h_K^{2r+1} \|\bm{\beta} \|_{0,\infty,K}^2 \| \bm{u} \|_{r+1,K}^2)^{\frac{1}{2}},\quad\forall \bm{v}_h \in \bm{V}_h,
\end{aligned}   
\end{equation}
which proves \eqref{eq: consistency err}.
\end{proof}

\begin{remark}
	If $\bm{\beta}$ were used instead of $\overline{\bm{\beta}}$ in defining $S_h^2(\cdot,\cdot)$, the estimate would involve $\|\bm{\beta}\|_{r+1,\infty,K}$. To avoid requiring high regularity on $\bm{\beta}$, we therefore define $S_h^2(\cdot,\cdot)$ using the piecewise constant approximation $\overline{\bm{\beta}}$.
\end{remark}

\subsection{{A priori} error estimate}
We now derive the \emph{a priori} error estimate for the LPS scheme~\eqref{scheme}. In practice, we are particularly interested in the global solution behavior on quasi-uniform meshes, even in the presence of boundary or internal layers. Thus, in the following analysis, we only consider the case where the mesh $\mathcal{T}_h$ is quasi-uniform.

\begin{theorem}[error estimate]
\label{thm:convergence order}
Let ${\mathcal{T}_h}$ be quasi-uniform, and let $\bm{V}_h = \bm{W}_h + \bm{B}_h$ be defined as in \eqref{def: fem space}. Suppose that assumptions~\eqref{cf positive} and~\eqref{coer:positive eig} hold.  
If the exact solution $\bm{u}$ belongs to $\bm{H}^{r+1}(\Omega)$, then the numerical solution $\bm{u}_h$ of the LPS scheme~\eqref{scheme} satisfies
\begin{equation}
\label{err estimate}
 \| \bm{u} - \bm{u}_h \|_h \leq C h^{r + \frac{1}{2}} \| \bm{u} \|_{H^{r+1}(\Omega)}.
\end{equation}
\end{theorem}

\begin{proof}
We begin by decomposing the total error into an interpolation error and a projection error in the standard manner
\begin{equation}
\| \bm{u} - \bm{u}_h \|_h \leq \| \bm{u} - j_h \bm{u} \|_h + \| \bm{u}_h - j_h \bm{u} \|_h.
\end{equation}
For the interpolation error,
we have
\begin{equation}
    \begin{aligned}
&\| \bm{u} - j_h \bm{u} \|_h^2 := \| \bm{u} - j_h \bm{u} \|^2_0  \\
&+ \underbrace{\sum_{f \in \mathcal{F}_{+}^{\partial}} \| \bm{u} - j_h \bm{u} \|^2_{f, \frac{1}{2} \bm{\bm{\beta}}} + \sum_{f \in \mathcal{F}^{\partial}_{-}} \| \bm{u} - j_h \bm{u} \|^2_{f, -\frac{1}{2} \bm{\bm{\beta}}} + \sum_{f \in \mathcal{F}^{\circ}} \| [\bm{u} - j_h \bm{u}]_f \|^2_{f, c_f \bm{\bm{\beta}}}}_{A_1}  \\
&+ \underbrace
{
\sum_{K \in \mathcal{T}_h} h_K \|\kappa_h (\mathcal{L}^{*}_{\overline{\bm{\beta}}} (\bm{u} - j_h \bm{u}))\|^2_{0,K}
}_{A_2}. \\
\end{aligned}
\end{equation}
Using the approximation property \eqref{eq:appro} of $j_h$, the trace inequality and the $L^2$ stability of $\kappa_h$, 
we deduce that
$\| \bm{u} - j_h \bm{u} \|_0 \leq C h^{r+1} |\bm{u}|_{r+1}$ and
\begin{equation}
\begin{aligned}
A_1
&\leq C (h^{-1} \| \bm{u} - j_h \bm{u} \|_{0,\Omega}^2 + h |\bm{u} - j_h \bm{u}|_{1,\Omega}^2) 
\leq C h^{2r+1} |\bm{u}|_{r+1,\Omega}^2,  \\
A_2
&\leq C\sum_{K \in \mathcal{T}_h} h_K \| \mathcal{L}^{*}_{\overline{\bm{\beta}}} (\bm{u} - j_h \bm{u}) \|_{0,K}^2 \\
&\leq \sum_{K \in \mathcal{T}_h} h_K \| \overline{\bm{\beta}} \|_{0,\infty,K}^2 \| \bm{u} - j_h \bm{u}\|_{1,K}^2 
\leq \sum_{K \in \mathcal{T}_h} h_K h_K^{2r} \| \overline{\bm{\beta}} \|_{0,\infty,K}^2 \| \bm{u} \|_{r+1,K}^2. \\
\end{aligned}
\end{equation}
Hence, we obtain the interpolation estimate
\begin{equation}
\label{err:appro}
\| \bm{u} - j_h \bm{u} \|_h \leq C (h^{r+1} + h^{r+1/2}) \|\bm{u}\|_{r+1,\Omega}\leq C h^{r+1/2} \|\bm{u}\|_{r+1,\Omega}.
\end{equation}
For the projection error, we apply Lemma \ref{lm:coercivity} (coercivity) to derive
\begin{align}
\| \bm{u}_h - j_h \bm{u} \|_h^2 \leq ~& C \Big[ a(\bm{u}_h - j_h \bm{u}, \bm{u}_h - j_h \bm{u}) + S_h^1(\bm{u}_h - j_h \bm{u}, \bm{u}_h - j_h \bm{u}) + S_h^2(\bm{u}_h - j_h \bm{u}, \bm{u}_h - j_h \bm{u}) \Big] \notag \\
=~& C \Big[\underbrace{a(\bm{u}_h - \bm{u}, \bm{u}_h - j_h \bm{u}) + S_h^1(\bm{u}_h - \bm{u}, \bm{u}_h - j_h \bm{u}) + S_h^2(\bm{u}_h - \bm{u}, \bm{u}_h - j_h \bm{u})}_{I_1} \notag \\
~&+ \underbrace{a(\bm{u} - j_h \bm{u}, \bm{u}_h - j_h \bm{u})}_{I_{2,a}} + \underbrace{S_h^1(\bm{u} - j_h \bm{u}, \bm{u}_h - j_h \bm{u}) + S_h^2(\bm{u} - j_h \bm{u}, \bm{u}_h - j_h \bm{u})}_{I_{2,S}}\Big]. \label{coer+garlerkin}
\end{align}

\underline{Estimate of $I_1$}: Using consistency error estimate \eqref{eq: consistency} and \eqref{eq: consistency err}, we obtain
\begin{align*}
|I_1| &= |-S_h^2(\bm{u}_h, \bm{u}_h - j_h \bm{u}) + S_h^2(\bm{u}_h - \bm{u}, \bm{u}_h - j_h \bm{u})| = |-S_h^2(\bm{u}, \bm{u}_h - j_h \bm{u})| \notag \\
&\leq C \| \bm{u}_h - j_h \bm{u} \|_h  \Big(\sum_{K \in \mathcal{T}_h} h_K^{2r+1}  \| {\bm{\beta}} \|^2_{0,\infty,K} \| \bm{u} \|_{r+1,K}^2\Big)^{1/2}
\leq C \| \bm{u}_h - j_h \bm{u} \|_h h^{r+1/2} \| \bm{u} \|_{r+1,\Omega}.  \label{consis err}
\end{align*}

\underline{Estimate of $I_{2,S}$}: The estimate follows directly from the Cauchy–Schwarz inequality and the interpolation estimate \eqref{err:appro}, namely,
$$
\begin{aligned}
S_h^1(\bm{u} - j_h \bm{u}, \bm{u}_h - j_h \bm{u})  +S_h^2(\bm{u} - j_h \bm{u},\bm{u}_h - j_h \bm{u}) & \leq C\|\bm{u} - j_h \bm{u}\|_h  \|\bm{u}_h - j_h \bm{u}\|_h \\
& \leq C h^{r+1/2} \|\bm{u}\|_{r+1,\Omega} \|\bm{u}_h - j_h \bm{u}\|_h.
\end{aligned}
$$

\underline{Estimate of $I_{2,a}$}: We write
\begin{equation*}
    \begin{aligned}
  I_{2,a}
  =~& (\gamma (\bm{u} - j_h \bm{u}), \bm{u}_h - j_h \bm{u})_\Omega   + \sum_{K \in \mathcal{T}_h} (\bm{u} - j_h \bm{u}, \mathcal{L}^{*}_{\bm{\beta}} (\bm{u}_h - j_h \bm{u}))_K \\
&+ \sum_{f \in \mathcal{F}^{\partial}_+} \left\langle \bm{u} - j_h \bm{u}, \bm{u}_h - j_h \bm{u}\right\rangle_{f, \bm{\bm{\beta}}}   + \sum_{f \in \mathcal{F}^{\circ}} \left\langle \{ \bm{u} - j_h \bm{u} \}_f , [\bm{u}_h - j_h \bm{u}]_f \right\rangle_{f, \bm{\bm{\beta}}} .                  \\
\end{aligned}
\end{equation*}
Applying the Cauchy–Schwarz inequality and trace inequality yields the following estimates:
\begin{equation}
\label{eq:I2-1}
   \begin{aligned}
   (\gamma (\bm{u} - j_h \bm{u}), \bm{u}_h - j_h \bm{u})_\Omega 
\leq ~& C \|\bm{u} - j_h \bm{u} \|_{0,\Omega} \| \bm{u}_h - j_h \bm{u}\|_h,\\
\sum_{f \in \mathcal{F}^{\partial}_+} \left\langle \bm{u} - j_h \bm{u}, \bm{u}_h - j_h \bm{u}\right\rangle_{f, \bm{\bm{\beta}}} 
\leq ~& C \sum_{f \in \mathcal{F}^{\partial}_+}\| \bm{u} - j_h \bm{u}\|_{0,f} \|\bm{u}_h - j_h \bm{u} \|_{f,\frac12\bm{\beta}} \\
\leq ~& C A_1^{\frac12} \|\bm{u}_h - j_h \bm{u} \|_h.
\end{aligned} 
\end{equation}
To estimate the error induced by the convection term across the element interfaces, we employ assumption \eqref{cf positive} together with $c_f = \mathcal{O}(1)$, and obtain:
\begin{equation}
\label{eq:I2-2}
\begin{aligned}
\sum_{f \in \mathcal{F}^{\circ}} \left\langle \{ \bm{u} - j_h \bm{u} \}_f , [\bm{u}_h - j_h \bm{u}]_f \right\rangle_{f, \bm{\bm{\beta}}}  
\leq ~& C \sum_{f \in \mathcal{F}^{\circ}}
\|\bm{u} - j_h \bm{u}\|_{0,f} \|[\bm{u}_h - j_h \bm{u}]\|_{f, c_f\bm{\beta}}\\
\leq ~& C A_1^{\frac12} \|\bm{u}_h - j_h \bm{u}\|_h.
\end{aligned}
\end{equation}
We now estimate the advection term within the elements. Using the approximation property of $\overline{\bm{\beta}}$ and the inverse inequality, we deduce that
\begin{equation}
\label{eq:I2-3}
\begin{aligned}
&|\sum_{K \in \mathcal{T}_h} (\bm{u} - j_h \bm{u}, 
\mathcal{L}^{*}_{\bm{\bm{\beta}}} (\bm{u}_h - j_h \bm{u})-\mathcal{L}^{*}_{\overline{\bm{\beta}}} (\bm{u}_h - j_h \bm{u}))_K|   \\
\leq & \sum_{K \in \mathcal{T}_h} \| \bm{u} - j_h \bm{u} \|_{0,K} (|\bm{u}_h-j_h\bm{u}|_{1,K} \|\bm{\beta}-{\overline{\bm{\beta}}}\|_{0,\infty,K}  + \|\bm{u}_h - j_h \bm{u}\|_{0,K} |\bm{\beta}-{\overline{\bm{\beta}}}|_{1,\infty,K})   \\
\leq&  \sum_{K \in \mathcal{T}_h} \| \bm{u} - j_h \bm{u} \|_{0,K} \|\bm{\beta}\|_{1,\infty,K} (h|\bm{u}_h-j_h\bm{u}|_{1,K}  + \|\bm{u}_h - j_h \bm{u}\|_{0,K})  \\
\leq&  \sum_{K \in \mathcal{T}_h} \| \bm{u} - j_h \bm{u} \|_{0,K} \|\bm{\beta}\|_{1,\infty,K} \|\bm{u}_h - j_h \bm{u}\|_{0,K}
\leq C \| \bm{u} - j_h \bm{u} \|_0 \| \bm{u}_h - j_h \bm{u} \|_h.   
\end{aligned}
\end{equation}
Moreover, the orthogonality property of $j_h$ established in \eqref{eq:orth} implies that
\begin{equation} \label{eq:I2-4}
   \begin{aligned}
|\sum_{K \in \mathcal{T}_h} (\bm{u} - j_h \bm{u}, \mathcal{L}^{*}_{\overline{\bm{\beta}}} (\bm{u}_h - j_h \bm{u}))_K | 
=&~|\sum_{K \in \mathcal{T}_h} (\bm{u} - j_h \bm{u}, \kappa_h (\mathcal{L}^{*}_{\overline{\bm{\beta}}} (\bm{u}_h - j_h \bm{u})))_K |  \\
=&~|\sum_{K \in \mathcal{T}_h} (h_K^{-\frac12} (\bm{u} - j_h \bm{u}), h_K^{\frac12} \kappa_h (\mathcal{L}^{*}_{\overline{\bm{\beta}}} (\bm{u}_h - j_h \bm{u})))_K |  \\
\leq&~ C h^{-\frac12} \| \bm{u} - j_h \bm{u} \|_{0,\Omega} \| \bm{u}_h - j_h \bm{u} \|_h.
\end{aligned} 
\end{equation}
Combining the estimates \eqref{eq:I2-1} -- \eqref{eq:I2-4}, we obtain
\begin{align}
|I_{2,a}| \leq C (h^{-\frac12} \| \bm{u} - j_h \bm{u} \|_0 + A_1^{\frac12}) \| \bm{u}_h - j_h \bm{u} \|_h \leq Ch^{r+\frac12} \|\bm{u}\|_{r+1,\Omega} \| \bm{u}_h - j_h \bm{u} \|_h . \label{err I2}
\end{align}
Therefore, it follows from the above estimates of $I_1$, $I_{2,S}$ and $I_{2,a}$ that
\begin{equation}
\label{err:u_h-j_hu}
\| \bm{u}_h - j_h \bm{u} \|_h \leq C h^{r+1/2} \| \bm{u} \|_{r+1}.
\end{equation}
Finally, combining \eqref{err:appro} and \eqref{err:u_h-j_hu}, we obtain the error estimate \eqref{err estimate}.
\end{proof}

\begin{remark}
	The proof of Theorem~\ref{thm:convergence order} highlights the distinct roles of the two stabilization terms, $S^1_h(\cdot,\cdot)$ and $S^2_h(\cdot,\cdot)$. Specifically, $S^1_h(\cdot,\cdot)$ primarily controls the contribution across the element interfaces, while $S^2_h(\cdot,\cdot)$ is crucial for handling the advection term within the elements, where the modified interpolation operator $j_h$ plays a central role. Their combined effect guarantees the optimal convergence rate of the method. Numerical experiments further confirm that both stabilization components are essential.
\end{remark}

\section{Numerical experiments}\label{sec:numerical}
In this section, we present numerical experiments for both two-dimensional and three-dimensional problems to validate our theoretical findings. All computations are carried out on uniform meshes with varying mesh sizes. The numerical results not only confirm the convergence rate predicted in Theorem~\ref{thm:convergence order}, but also underscore the essential role of bubble enrichment and the stabilization terms $S_h^1(\cdot,\cdot)$ and $S_h^2(\cdot,\cdot)$ in achieving optimal convergence. In addition, we include test cases that demonstrate the robustness of the method in effectively suppressing spurious oscillations in solutions exhibiting layers.

In cases with translational symmetry, the three-dimensional formulations \eqref{eq:3d h curl} and \eqref{eq:3d h div} reduce to the following two-dimensional boundary value problems posed on a domain $\Omega \subset \mathbb{R}^2$:
\begin{equation}
\label{2d curl eqn}
						\left\{
						\begin{aligned}
		                -\bm{R}\bm{\bm{\beta}} (\nabla \cdot \bm{R}\bm{u})		
							 + \nabla (\bm{\bm{\beta}} \cdot \bm{u}) +  \gamma
							\bm{u}&= \bm{f} ~~\quad &&\text{in }\Omega, \\
							\bm{u} 
							&= \bm{g} \quad &&\text{on } \Gamma_{-}, \\
							\end{aligned}
						\right.
\end{equation} 
 \begin{equation}
              \label{2d div eqn}
						\left\{
						\begin{aligned}
		                \bm{\beta} (\nabla \cdot \bm{u}) + \bm{R}\nabla (\bm{u} \times  \bm{\beta})+  \gamma
							\bm{u}&= \bm{f} ~~\quad &&\text{in }\Omega, \\
							\bm{u} 
							&= \bm{g} \quad &&\text{on } \Gamma_{-}, \\
							\end{aligned}
						\right.
\end{equation}
where \( \bm{R} = \begin{bmatrix} 0 & 1 \\ -1 & 0 \end{bmatrix} \) is the ${\pi}\over{2}$-rotation matrix and the two-dimensional cross product is defined as
$\bm{u} \times  \bm{\beta}:=\det([\bm{u} ~  \bm{\beta}])$. 

Equations~\eqref{2d curl eqn} and~\eqref{2d div eqn} are, in fact, equivalent up to a simple transformation. In the two-dimensional setting, let ${\lambda}_0$, ${\lambda}_1$, and ${\lambda}_2$ denote the barycentric coordinates associated with the three vertices of an element $K$. We define
\begin{equation}
    \label{2d curl bubble}
    \bm{B}_h^{\mathrm{curl}}(K):=P_{r-1}(K)\bm{b}_1^{\rm curl}+P_{r-1}(K) \bm{b}_2^{\rm curl}, \quad\forall K \in \mathcal{T}_h,
\end{equation}
where ${\bm{b}}_1^{\rm curl} := {\lambda}_2 {\lambda}_0{\bm{n}}_1$, 
${\bm{b}}_2^{\rm curl} := {\lambda}_0 {\lambda}_1 {\bm{n}}_2$.

Similarly, for the divergence case, we define
\begin{equation}
    \label{2d div bubble}
    \bm{B}_h^{\mathrm{div}}(K):=P_{r-1}(K)\bm{b}_1^{\rm div}+P_{r-1}(K) \bm{b}_2^{\rm div},\quad \forall K \in \mathcal{T}_h,
\end{equation}
where ${\bm{b}}_1^{\rm div} := {\lambda}_0 {\lambda}_1 \bm{t}_{01}$, 
${\bm{b}}_2^{\rm div} := {\lambda}_0 {\lambda}_2 \bm{t}_{02}$.

The constructions and theoretical results presented in the previous sections, particularly Theorem~\ref{thm: existence of j_h} and \ref{thm:convergence order}, remain valid in this reduced two-dimensional setting.
 
 \subsection{Example 1: 2D convergence tests}
\label{2d}
Let $\Omega = (0, 1)^2$, $\gamma = 1$,  and set
\[
\bm{\beta}(x, y) = \begin{bmatrix}
y - {1 \over 2} \\
-x+{1\over2} 
\end{bmatrix}.
\]
so that condition~\eqref{coer:positive eig} is satisfied. We set $c_f = \text{sgn}(\bm{\beta}\cdot\bm{n})$ to ensure condition~\eqref{cf positive}. For both equations~\eqref{2d curl eqn} and~\eqref{2d div eqn}, the source terms $\bm{f}$ and $\bm{g}$ are chosen such that the exact solution is given by
\[
\bm{u}(x, y) = \begin{bmatrix}
\sin(x) \cdot \cos(y) \\
\exp(x) \cdot y^2
\end{bmatrix}.
\]

\begin{table}[!htbp]
\centering
\caption{Example 1: Convergence results for the 2D $\bm{H}(\mathrm{curl})$ problem with a smooth solution for $r=1,2$.}
\label{tab:curl_2d}
\setlength{\tabcolsep}{4mm}  
\begin{tabular}{c c | c c c c | c c}
\toprule
\multirow{2}{*}{$r$} & \multirow{2}{*}{$1/h$} & 
\multicolumn{4}{c|}{LPS} & 
\multicolumn{2}{c}{No Stabilization} \\
\cmidrule(lr){3-6} \cmidrule(lr){7-8}
& & 
$\| \cdot \|_h$ error & 
order & 
$L^2$ error & 
order & 
$L^2$ error & 
order \\
\midrule
\multirow{6}{*}{1}
& 4   & 7.329e-02 & -- & 1.873e-02 & -- & 2.437e-02 & -- \\
& 8   & 2.573e-02 & 1.51 & 4.866e-03 & 1.94 & 9.910e-03 & 1.30 \\
& 16  & 9.016e-03 & 1.51 & 1.251e-03 & 1.96 & 4.256e-03 & 1.22 \\
& 32  & 3.166e-03 & 1.51 & 3.196e-04 & 1.97 & 1.891e-03 & 1.17 \\
& 64  & 1.115e-03 & 1.51 & 8.116e-05 & 1.98 & 8.596e-04 & 1.14 \\
& 128 & 3.935e-04 & 1.50 & 2.053e-05 & 1.98 & 4.068e-04 & 1.08 \\
\midrule
\multirow{6}{*}{2}
& 4   & 6.616e-03 & -- & 2.014e-03 & -- & 7.962e-04 & -- \\
& 8   & 1.124e-03 & 2.56 & 2.753e-04 & 2.87 & 1.403e-04 & 2.50 \\
& 16  & 1.937e-04 & 2.54 & 3.654e-05 & 2.91 & 2.741e-05 & 2.36 \\
& 32  & 3.360e-05 & 2.53 & 4.880e-06 & 2.90 & 4.738e-06 & 2.53 \\
& 64  & 5.870e-06 & 2.52 & 6.439e-07 & 2.92 & 8.352e-07 & 2.50 \\
& 128 & 1.017e-06 & 2.53 & 8.476e-08 & 2.93 & 1.377e-07 & 2.60 \\
\bottomrule
\end{tabular}
\end{table}

\begin{table}[h!]
\centering
\caption{Example 1: Convergence results for the 2D $\bm{H}(\mathrm{div})$ problem with a smooth solution for $r=1,2$.}
\label{tab:div_2d}
\setlength{\tabcolsep}{4mm}  
\begin{tabular}{c c | c c c c | c c}
\toprule
\multirow{2}{*}{$r$} & \multirow{2}{*}{$1/h$} & 
\multicolumn{4}{c|}{LPS} & 
\multicolumn{2}{c}{No Stabilization} \\
\cmidrule(lr){3-6} \cmidrule(lr){7-8}
& & 
$\| \cdot \|_h$ error & 
order & 
$L^2$ error & 
order & 
$L^2$ error & 
order \\
\midrule
\multirow{6}{*}{1}
& 4   & 5.823e-02 & -- & 1.744e-02 & --& 2.112e-02 & -- \\
& 8   & 1.986e-02 & 1.55 & 4.379e-03 & 1.99 & 8.352e-03 & 1.34 \\
& 16  & 6.905e-03 & 1.52 & 1.111e-03 & 1.98 & 3.275e-03 & 1.35 \\
& 32  & 2.418e-03 & 1.51 & 2.818e-04 & 1.98 & 1.317e-03 & 1.31 \\
& 64  & 8.502e-04 & 1.51 & 7.133e-05 & 1.98 & 5.159e-04 & 1.35 \\
& 128 & 2.998e-04 & 1.50 & 1.808e-05 & 1.98 & 2.084e-04 & 1.31 \\
\midrule
\multirow{6}{*}{2}
& 4   & 6.604e-03 & -- & 1.802e-03 & -- & 7.815e-04 & -- \\
& 8   & 1.135e-03 & 2.54 & 2.458e-04 & 2.87 & 1.426e-04 & 2.45 \\
& 16  & 1.962e-04 & 2.53 & 3.133e-05 & 2.97 & 2.777e-05 & 2.36 \\
& 32  & 3.437e-05 & 2.51 & 3.974e-06 & 2.98 & 5.022e-06 & 2.47 \\
& 64  & 6.056e-06 & 2.50 & 5.029e-07 & 2.98 & 8.948e-07 & 2.49 \\
& 128 & 1.064e-06 & 2.51 & 6.371e-08 & 2.98 & 1.475e-07 & 2.60 \\
\bottomrule
\end{tabular}
\end{table}

For equation~\eqref{2d curl eqn}, we take the bubble space $\bm{B}_h^{\mathrm{curl}}(K)$ as defined in~\eqref{2d curl bubble}, with polynomial orders $r = 1, 2$. The underlying finite element space $\bm{W}_h$ is chosen as the Nédélec element of the second kind of order~$r$ (in 2D, this corresponds to a $\frac{\pi}{2}$-rotation of the BDM element), so that the approximation property~\eqref{W_h:i_h_appro} is satisfied. For $r = 1$, the bubble enrichment adds 2 local degrees of freedom per element. For $r = 2$, the enrichment contributes 4 local degrees of freedom per element, since $\bm{W}_h$ already contains certain $\bm{H}(\mathrm{curl})$-conforming bubble functions. Similarly, for equation~\eqref{2d div eqn}, we define the bubble space $\bm{B}_h^{\mathrm{div}}(K)$ as in~\eqref{2d div bubble}, again for $r = 1, 2$. The space $\bm{W}_h$ is taken to be the BDM element of order~$r$.

Tables~\ref{tab:curl_2d} and~\ref{tab:div_2d} summarize the convergence behavior in both the energy norm $\| \cdot \|_h$ and the $L^2$ norm with respect to the mesh size $h$. For the LPS method, we observe $(r + \tfrac{1}{2})$-th order convergence in the energy norm~\eqref{ene norm}, in agreement with the theoretical prediction in Theorem~\ref{thm:convergence order}. Additionally, the results indicate an observed $(r + 1)$-th order convergence in the $L^2$ norm.

For comparison, the performance of the scheme without stabilization is also included. The contrast between the stabilized and unstabilized results clearly demonstrates the effectiveness of the proposed stabilization strategy in achieving the expected convergence rates and improving overall accuracy.

\subsection{Example 2: 3D convergence tests}
\label{3d experiment}
We now turn to the three-dimensional examples to further validate the proposed method and confirm its effectiveness with the LPS stabilization and bubble enrichment.

Let $\Omega = (0, 1)^3$ and set $\gamma = 8$. We define the advection field as
\begin{equation*}
\bm{\beta}(x, y, z) = \begin{bmatrix}
\exp(x^2) \\
y \sin(z) \\
x y z
\end{bmatrix},
\end{equation*}
which satisfies condition~\eqref{coer:positive eig}. As before, we set $c_f = \operatorname{sgn}(\bm{\beta} \cdot \bm{n})$ to ensure that condition~\eqref{cf positive} holds.

The source terms $\bm{f}$ and $\bm{g}$ are chosen so that the following smooth vector field serves as the exact solution to~\eqref{eq:3d unified}
\begin{equation*}
\bm{u}(x, y, z) = \begin{bmatrix}
y \exp(x z) \\
- x y \\
\sin(x y z)
\end{bmatrix}.
\end{equation*}

For both $\mathrm{d} = \mathrm{curl}$ and $\mathrm{d} = \mathrm{div}$, the bubble enrichment contributes 3 local degrees of freedom per element when $r = 1$. Tables~\ref{tab:3d curl} and~\ref{tab:3d div} report the convergence rates in the energy norm $\| \cdot \|_h$ and the $L^2$ norm for $r = 1$, with respect to the mesh size $h$. For the LPS method, we observe a convergence rate of $\frac32$ in the energy norm~\eqref{ene norm}, in agreement with the theoretical prediction in Theorem~\ref{thm:convergence order}. Additionally, due to the smoothness of the exact solution, second-order convergence is observed in the $L^2$ norm. 

For comparison, results from the scheme without stabilization are also included. The contrast clearly illustrates the stabilizing effect of the proposed method, confirming its effectiveness in improving accuracy and ensuring the predicted convergence behavior.

\begin{table}[!htbp]
\centering
\caption{Example 2: Convergence results for the 3D $\bm{H}(\mathrm{curl})$ problem with a smooth solution for $r=1$.}
\label{tab:3d curl}
\setlength{\tabcolsep}{6mm}  
\begin{tabular}{c | cc cc | cc cc | cc}
\toprule
\multirow{2}{*}{$1/h$} & 
\multicolumn{4}{c|}{LPS } & 
\multicolumn{2}{c}{No Stabilization} \\
\cmidrule(lr){2-5}  \cmidrule(lr){6-7}
& 
\multicolumn{1}{c}{$\| \cdot \|_h$ error} & 
\multicolumn{1}{c}{order} & 
\multicolumn{1}{c}{$L^2$ error} & 
\multicolumn{1}{c}{order} & 
\multicolumn{1}{c}{$L^2$ error} & 
\multicolumn{1}{c}{order} \\
\midrule
1  & 1.869e+0 & --    & 1.877e-1 & --    & 1.086e-1 & --    \\
2  & 7.141e-1 & 1.39  & 7.376e-2 & 1.35  & 4.309e-2 & 1.33  \\
4  & 2.408e-1 & 1.57  & 2.039e-2 & 1.86  & 1.512e-2 & 1.51  \\
8  & 8.280e-2 & 1.54  & 4.593e-3 & 2.15  & 5.356e-3 & 1.50  \\
16 & 2.909e-2 & 1.51  & 1.069e-3 & 2.10  & 1.977e-3 & 1.44  \\
32 & 1.029e-2 & 1.50  & 2.601e-4 & 2.04  & 7.503e-4 & 1.40  \\
\bottomrule
\end{tabular}
\end{table}

\begin{table}[!htbp]
\centering
\caption{Example 2: Convergence results for the 3D $\bm{H}(\mathrm{div})$ problem with a smooth solution for $r=1$.}
\label{tab:3d div}
\setlength{\tabcolsep}{6mm}  
\begin{tabular}{c | cc cc | cc}
\toprule
\multirow{2}{*}{$1/h$} & 
\multicolumn{4}{c|}{LPS} &  
\multicolumn{2}{c}{No Stabilization} \\
\cmidrule(lr){2-5} \cmidrule(lr){6-7}
& 
\multicolumn{1}{c}{$\| \cdot \|_h$ error} & 
\multicolumn{1}{c}{order} & 
\multicolumn{1}{c}{$L^2$ error} & 
\multicolumn{1}{c}{order} & 
\multicolumn{1}{c}{$L^2$ error} & 
\multicolumn{1}{c}{order} \\
\midrule
1  & 2.108e+0 & --  & 2.301e-1 & --  & 1.038e-1 & --  \\
2  & 8.849e-1 & 1.25  & 9.536e-2 & 1.27  & 3.418e-2 & 1.60  \\
4  & 2.989e-1 & 1.57  & 2.637e-2 & 1.85  & 1.136e-2 & 1.59  \\
8  & 1.021e-1 & 1.55  & 6.109e-3 & 2.11  & 3.658e-3 & 1.64  \\
16 & 3.567e-2 & 1.52  & 1.427e-3 & 2.10  & 1.215e-3 & 1.59  \\
32 & 1.262e-2 & 1.50  & 3.488e-4 & 2.03  & 4.206e-4 & 1.53  \\
\bottomrule
\end{tabular}
\end{table}



\subsection{Example 3: Clarification of stabilization terms}
To further clarify the role of the stabilization terms, we present additional convergence results that demonstrate the necessity of including both $S_h^1(\cdot,\cdot)$ and $S_h^2(\cdot,\cdot)$ in order to achieve optimal convergence rates. We consider the case $\mathrm{d} = \mathrm{curl}$ and adopt the same setup as in Example~\ref{3d experiment}. The $S_h^1$-norm is defined  by
\begin{equation}
    \|\bm{u}\|_{S_h^1}^2 := \|\bm{u}\|^2_{L^2(\Omega)}  + \sum_{f \in \mathcal{F}^{\partial}_+} \|\bm{u}\|^2_{f, \frac{1}{2} \bm{\bm{\beta}}} + \sum_{f \in \mathcal{F}^{\partial}_{-}} \|\bm{u}\|^2_{f, -\frac{1}{2} \bm{\bm{\beta}}} + \sum_{f \in \mathcal{F}^{\circ}} \|[\bm{u}]_f\|^2_{f, c_f \bm{\bm{\beta}}}.
\end{equation} 
Table~\ref{tab:3d curl S1+S2 is necessary} compares three variants of the method: full stabilization using both $S_h^1(\cdot,\cdot)$ and $S_h^2(\cdot,\cdot)$, partial stabilization using only $S_h^1(\cdot,\cdot)$, and partial stabilization using only $S_h^2(\cdot,\cdot)$. As shown in the table, omitting the $S_h^2(\cdot,\cdot)$ term results in a reduction in convergence order, while neglecting the $S_h^1(\cdot,\cdot)$ term leads to a non-convergent scheme. These results confirm that both stabilization terms are essential for achieving the optimal convergence predicted by the theory.

\begin{table}[!htbp]
\centering
\caption{Example 3: Errors and convergence rates for the 3D $\bm{H}(\mathrm{curl})$ problem with $r = 1$ and various combinations of stabilization terms.}

\label{tab:3d curl  S1+S2 is necessary}
\begin{tabular}{c | cc cc | cc cc | cc}
\toprule
\multirow{2}{*}{$1/h$} & 
\multicolumn{4}{c|}{$S_h^1+S_h^2$ } & 
\multicolumn{4}{c|}{$S_h^1$ only} & 
\multicolumn{2}{c}{$S_h^2$ only} \\
\cmidrule(lr){2-5} \cmidrule(lr){6-9} \cmidrule(lr){10-11}
& 
\multicolumn{1}{c}{$\| \cdot \|_h$ error} & 
\multicolumn{1}{c}{order} & 
\multicolumn{1}{c}{$L^2$ error} & 
\multicolumn{1}{c}{order} & 
\multicolumn{1}{c}{$\| \cdot \|_{S_h^1}$ error} & 
\multicolumn{1}{c}{order} & 
\multicolumn{1}{c}{$L^2$ error} & 
\multicolumn{1}{c}{order} & 
\multicolumn{1}{c}{$L^2$ error} & 
\multicolumn{1}{c}{order} \\
\midrule
 1 & 1.869e+0 & -- & 1.877e-1 & -- & 2.282e-1 & -- & 1.078e-1 & -- & 2.465e-1 & -- \\
2 & 7.141e-1 & 1.39 & 7.376e-2 & 1.35 & 8.891e-2 & 1.36 & 4.075e-2 & 1.40 & 2.231e-1 & 0.14 \\
4 & 2.408e-1 & 1.57 & 2.039e-2 & 1.86 & 3.305e-2 & 1.43 & 1.282e-2 & 1.67 & 2.658e-1 & -0.25 \\
8 & 8.280e-2 & 1.54 & 4.593e-3 & 2.15 & 1.245e-2 & 1.41 & 3.945e-3 & 1.70 & 3.050e-1 & -0.20 \\
16 & 2.909e-2 & 1.51 & 1.069e-3 & 2.10 & 4.775e-3 & 1.38 & 1.241e-3 & 1.67 & 3.308e-1 & -0.12 \\
32 & 1.029e-2 & 1.50 & 2.601e-4 & 2.04 & 1.838e-3 & 1.38 & 3.887e-4 & 1.68 & 3.467e-1 & -0.07 \\
\bottomrule
\end{tabular}
\end{table}

\subsection{Example 4: Clarification of bubble enrichment}
We now present an example to demonstrate that bubble enrichment is essential for achieving the optimal convergence order. In this example, we consider the case $\mathrm{d} = \mathrm{curl}$. We take $\bm{W}_h$ to be the Nédélec element of the second kind of order $2$ and set $\bm{V}_h = \bm{W}_h$, using the same problem data as in Example~\ref{2d}. 

In this example without bubble enrichment, the stabilization term $S_h^2(\cdot,\cdot)$ has no effect. Therefore, we only compare the cases with and without the inclusion of $S^1_h(\cdot,\cdot)$.
As shown in Table~\ref{tab:3d curl enrichment is necessary}, a significant reduction in convergence order is observed when the $\bm{H}(\mathrm{curl})$ bubble enrichment is omitted. This contrasts with the results in Table~\ref{tab:curl_2d}, where the optimal convergence orders of 2.5 in the energy norm and 3 in the $L^2$ norm are achieved, and confirms that the enrichment is indispensable for obtaining optimal rates.

\begin{table}[!htbp]
\centering
\caption{Example 4: Errors and convergence rates for the 2D $\bm{H}(\mathrm{curl})$ problem using $r = 2$ without bubble enrichment.}
\label{tab:3d curl  enrichment is necessary}
\begin{tabular}{c| cc cc | cc}
\toprule
\multirow{2}{*}{$1/h$} & 
\multicolumn{4}{c|}{$S_h^1$ only} & 
\multicolumn{2}{c}{No Stabilization} \\
\cmidrule(lr){2-5} \cmidrule(lr){6-7} 
&  
\multicolumn{1}{c}{$\| \cdot \|_{S_h^1}$ error} & 
\multicolumn{1}{c}{order} & 
\multicolumn{1}{c}{$L^2$ error} & 
\multicolumn{1}{c}{order} & 
\multicolumn{1}{c}{$L^2$ error} & 
\multicolumn{1}{c}{order} \\
\midrule
    4 & 2.725e-03 & -- & 2.267e-03 & -- & 1.132e-03 & -- \\
8 & 4.443e-04 & 2.62 & 2.906e-04 & 2.96 & 2.056e-04 & 2.46 \\
16 & 8.261e-05 & 2.43 & 3.934e-05 & 2.88 & 4.507e-05 & 2.19 \\
32 & 1.702e-05 & 2.28 & 6.863e-06 & 2.52 & 1.079e-05 & 2.06 \\
64 & 3.736e-06 & 2.19 & 1.546e-06 & 2.15 & 2.355e-06 & 2.20 \\
128 & 8.500e-07 & 2.14 & 4.004e-07 & 1.95 & 5.777e-07 & 2.03 \\
\bottomrule
\end{tabular}
\end{table}

\subsection{Example 5: 2D circular layer test}
We now present an example featuring a circular interior layer. We use the same spaces $\bm{B}_h^{\mathrm{curl}}$ and $\bm{W}_h$ as in section~\ref{2d}, with polynomial degree $r = 1$. The parameters $\gamma$ and $\bm{\beta}$ are also taken as in section~\ref{2d}. The source terms $\bm{f}$ and $\bm{g}$ are chosen so that the exact solution is
\[
\bm{u}(x, y) = \begin{bmatrix}
16x(1 - x)y(1 - y) \left( \frac{1}{2} + \frac{1}{\pi} \arctan\left( \frac{200}{0.25^2 - (x - 0.5)^2 - (y - 0.5)^2} \right) \right) \\
16x(1 - x)y(1 - y) \left( \frac{1}{2} + \frac{1}{\pi} \arctan\left( \frac{200}{0.25^2 - (x - 0.5)^2 - (y - 0.5)^2} \right) \right)
\end{bmatrix}.
\]

This solution exhibits a circular interior layer along the circumference of the circle centered at $(0.5, 0.5)$ with radius $0.25$, fully contained within the unit square domain $\Omega = (0,1)^2$.

Figure~\ref{trisurf_Circular} displays the first component of the LPS solution on a mesh with $1/h = 32$. The numerical solution accurately captures the interior layer and remains free of spurious oscillations.

\begin{figure}[!htbp]
    \centering
    \includegraphics[width=0.5\textwidth]{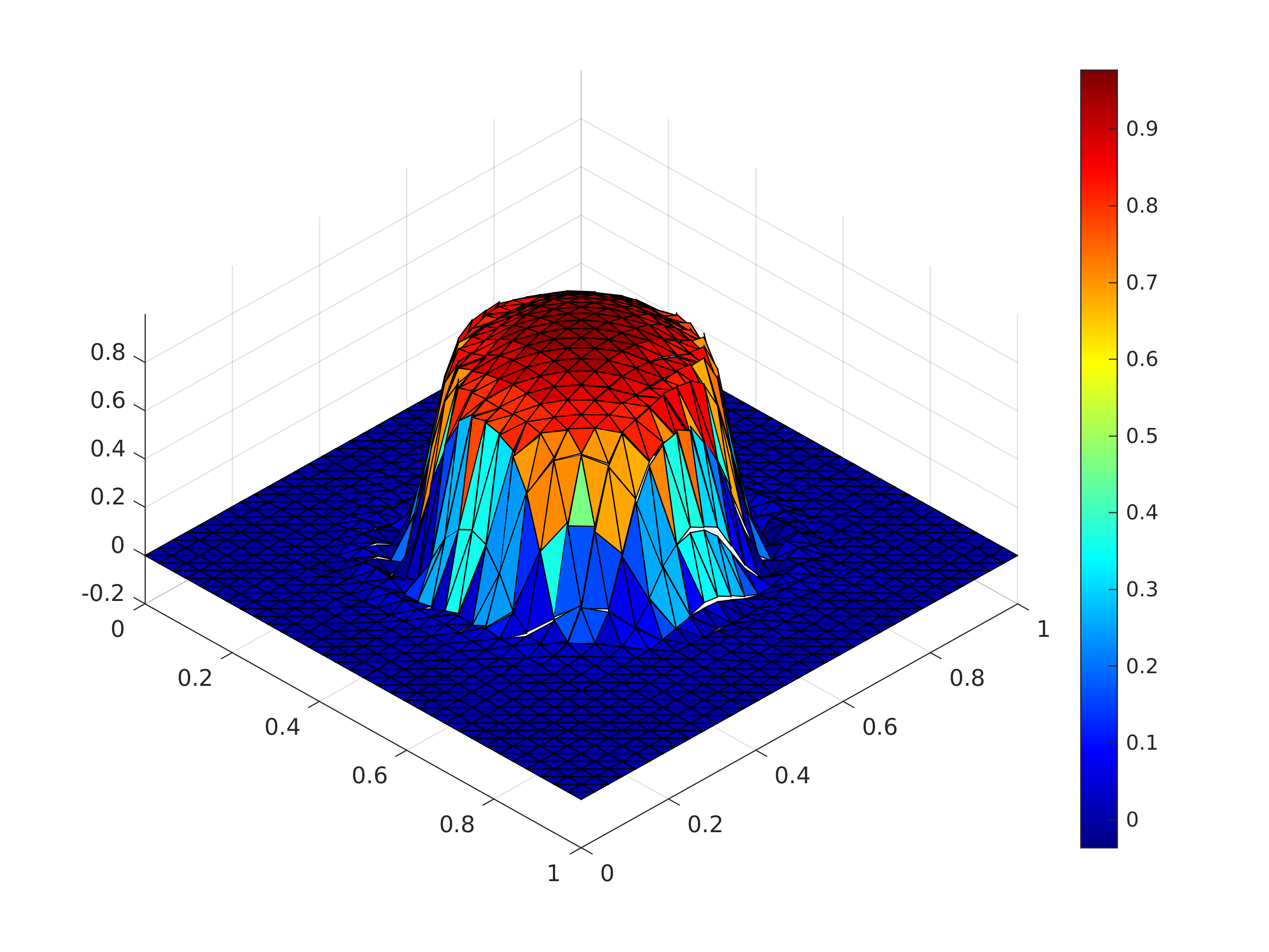}
    \caption{\centering Example 5: First component of the LPS solution featuring a circular interior layer.}
    \label{trisurf_Circular}
\end{figure}

\subsection{Example 6: 2D interior layer test}
Let $\Omega = (0, 1)^2$, $r = 1$, $\gamma = 0$, and $\bm{f} = \bm{0}$. We define the advection field and boundary data as
\[
\bm{\beta}(x, y) = \begin{bmatrix}
- y - 1 \\
x + 1
\end{bmatrix}, \qquad
\bm{g}(x, y) = \begin{cases}
(0, 0)^\top & \text{if } x > 0.6, \\
(1, 1)^\top & \text{otherwise}.
\end{cases}
\]

We consider equation~\eqref{2d curl eqn} and use the same finite element spaces as in section~\ref{2d}. Figure~\ref{fig:trisurf} shows the first component of the LPS solution on a mesh with $1/h = 32$. The interior layer is clearly resolved without visible numerical oscillations, further demonstrating the robustness and efficiency of the proposed method.

\begin{figure}[!htbp]
    \centering
    \includegraphics[width=0.5\textwidth]{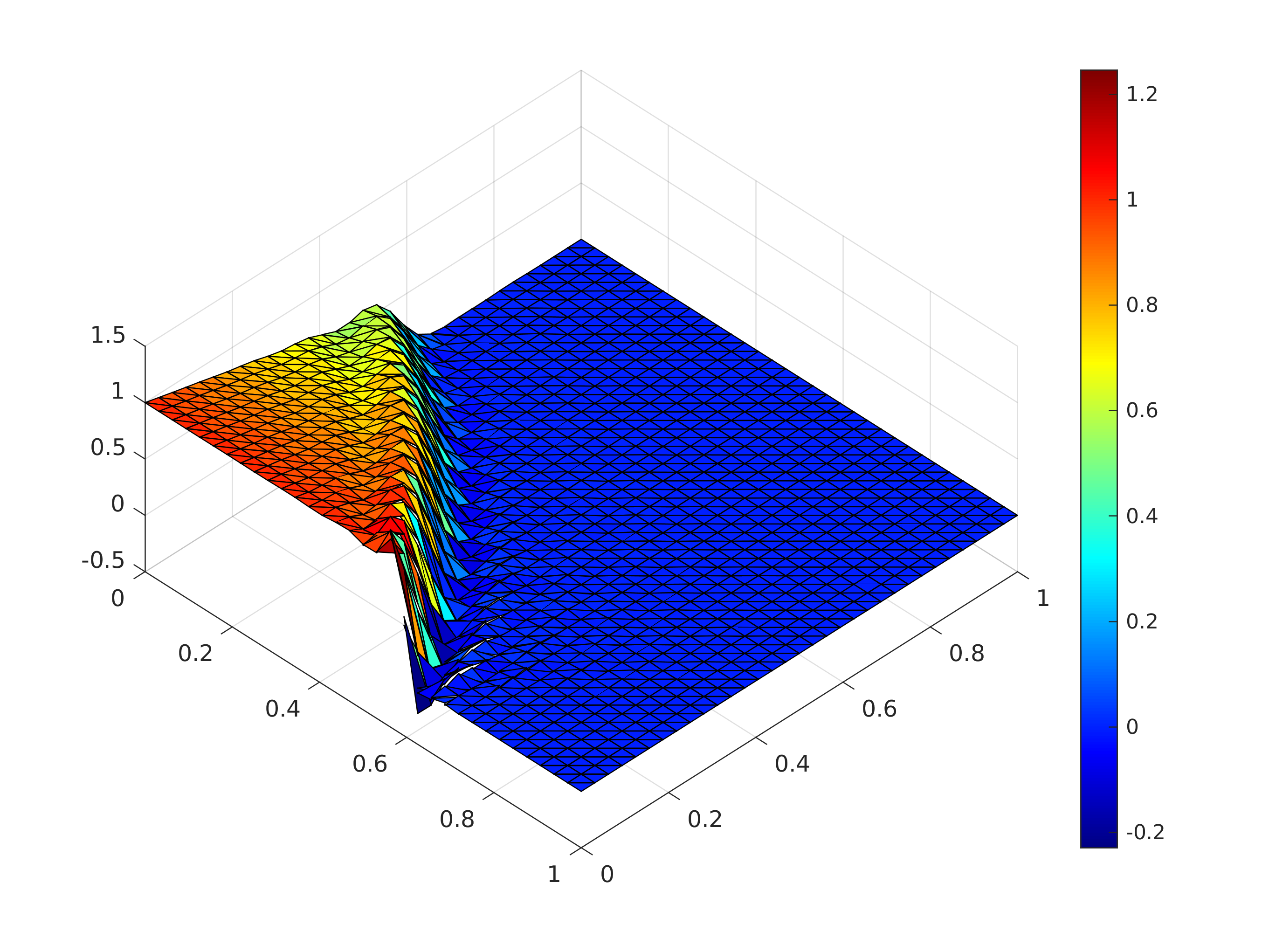}
    \caption{\centering Example 6: First component of the LPS solution featuring an interior layer.}
    \label{fig:trisurf}
\end{figure}

\section{Conclusion}\label{sec:conclusion}
In this paper, we have developed local projection stabilization (LPS) methods for $\bm{H}(\mathrm{curl})$ and $\bm{H}(\mathrm{div})$ pure advection equations. The proposed approach is applicable to $\bm{H}(\mathrm{d})$-conforming finite element spaces of arbitrary order, where $\mathrm{d} = \mathrm{curl}$ or $\mathrm{div}$. The resulting schemes possess stability and conformity properties that make them well-suited for potential applications to convection-dominated discretizations arising in magnetohydrodynamics (MHD).

A key component of the method is the establishment of a local inf-sup condition, achieved through enrichment of the approximation space with $\bm{H}(\mathrm{d})$ bubble functions. This enrichment facilitates the construction of a modified interpolation operator, which plays a central role in ensuring the stability and accuracy of the scheme. As demonstrated by the numerical experiments, both the bubble enrichment and the LPS stabilization terms are essential to achieve optimal convergence rates.
\section*{Acknowledgement} 
The work of Shuonan Wu is supported in part by the National Natural Science Foundation of China grant No.~12222101 and the Beijing Natural Science Foundation No.~1232007.

\bibliographystyle{plain}
\bibliography{lps}

\begin{thebibliography}{10}

\bibitem{ayuso2009discontinuous}
Blanca Ayuso and L~Donatella Marini.
\newblock Discontinuous {G}alerkin methods for advection-diffusion-reaction
  problems.
\newblock {\em SIAM Journal on Numerical Analysis}, 47(2):1391--1420, 2009.

\bibitem{becker2001finite}
Roland Becker and Malte Braack.
\newblock A finite element pressure gradient stabilization for the stokes
  equations based on local projections.
\newblock {\em Calcolo}, 38(4):173--199, 2001.

\bibitem{becker2004two}
Roland Becker and Malte Braack.
\newblock A two-level stabilization scheme for the {N}avier-{S}tokes equations.
\newblock In {\em Numerical Mathematics and Advanced Applications: Proceedings
  of ENUMATH 2003 the 5th European Conference on Numerical Mathematics and
  Advanced Applications Prague, August 2003}, pages 123--130. Springer, 2004.

\bibitem{boffi2013mixed}
Daniele Boffi, Franco Brezzi, and Michel Fortin.
\newblock {\em Mixed finite element methods and applications}, volume~44.
\newblock Springer-Verlag, Berlin, 2013.

\bibitem{braack2009optimal}
Malte Braack.
\newblock Optimal control in fluid mechanics by finite elements with symmetric
  stabilization.
\newblock {\em SIAM Journal on Control and Optimization}, 48(2):672--687, 2009.

\bibitem{braack2006local}
Malte Braack and Erik Burman.
\newblock Local projection stabilization for the {O}seen problem and its
  interpretation as a variational multiscale method.
\newblock {\em SIAM Journal on Numerical Analysis}, 43(6):2544--2566, 2006.

\bibitem{brenner2008mathematical}
Susanne~C Brenner.
\newblock {\em The mathematical theory of finite element methods}.
\newblock Springer, New York City, 2008.

\bibitem{brezzi1998further}
Franco Brezzi, Leopoldo~P Franca, and Alessandro Russo.
\newblock Further considerations on residual-free bubbles for
  advective-diffusive equations.
\newblock {\em Computer Methods in Applied Mechanics and Engineering},
  166(1-2):25--33, 1998.

\bibitem{brezzi1999priori}
Franco Brezzi, Thomas~JR Hughes, LD~Marini, Alessandro Russo, and Endre
  S{\"u}li.
\newblock A priori error analysis of residual-free bubbles for
  advection-diffusion problems.
\newblock {\em SIAM Journal on Numerical Analysis}, 36(6):1933--1948, 1999.

\bibitem{brezzi1998applications}
Franco Brezzi, D~Marini, and Alessandro Russo.
\newblock Applications of the pseudo residual-free bubbles to the stabilization
  of convection-diffusion problems.
\newblock {\em Computer Methods in Applied Mechanics and Engineering},
  166(1-2):51--63, 1998.

\bibitem{brezzi2004discontinuous}
Franco Brezzi, L~Donatella Marini, and Endre S{\"u}li.
\newblock Discontinuous {G}alerkin methods for first-order hyperbolic problems.
\newblock {\em Mathematical Models and Methods in Applied Sciences},
  14(12):1893--1903, 2004.

\bibitem{brezzi1994choosing}
Franco Brezzi and Alessandro Russo.
\newblock Choosing bubbles for advection-diffusion problems.
\newblock {\em Mathematical Models and Methods in Applied Sciences},
  4(04):571--587, 1994.

\bibitem{burman2005unified}
Erik Burman.
\newblock A unified analysis for conforming and nonconforming stabilized finite
  element methods using interior penalty.
\newblock {\em SIAM Journal on Numerical Analysis}, 43(5):2012--2033, 2005.

\bibitem{burman2010consistent}
Erik Burman.
\newblock Consistent {SUPG}-method for transient transport problems: stability
  and convergence.
\newblock {\em Computer Methods in Applied Mechanics and Engineering},
  199(17-20):1114--1123, 2010.

\bibitem{burman2007continuous}
Erik Burman and Alexandre Ern.
\newblock Continuous interior penalty $hp$-finite element methods for advection
  and advection-diffusion equations.
\newblock {\em Mathematics of Computation}, 76(259):1119--1140, 2007.

\bibitem{burman2009weighted}
Erik Burman, Johnny Guzm{\'a}n, and Dmitriy Leykekhman.
\newblock Weighted error estimates of the continuous interior penalty method
  for singularly perturbed problems.
\newblock {\em IMA Journal of Numerical Analysis}, 29(2):284--314, 2009.

\bibitem{chen2005optimal}
Long Chen and Jinchao Xu.
\newblock An optimal streamline diffusion finite element method for a
  singularly perturbed problem.
\newblock {\em Contemporary Mathematics}, 383:191--202, 2005.

\bibitem{franca1992stabilized}
Leopoldo~P Franca, Sergio~L Frey, and Thomas~JR Hughes.
\newblock Stabilized finite element methods: I. application to the
  advective-diffusive model.
\newblock {\em Computer Methods in Applied Mechanics and Engineering},
  95(2):253--276, 1992.

\bibitem{franca2002stability}
Leopoldo~P Franca and Lutz Tobiska.
\newblock Stability of the residual free bubble method for bilinear finite
  elements on rectangular grids.
\newblock {\em IMA Journal of Numerical Analysis}, 22(1):73--87, 2002.

\bibitem{friedrichs1958symmetric}
Kurt~O Friedrichs.
\newblock Symmetric positive linear differential equations.
\newblock {\em Communications on Pure and Applied Mathematics}, 11(3):333--418,
  1958.

\bibitem{ganesan2010stabilization}
Sashikumaar Ganesan and Lutz Tobiska.
\newblock Stabilization by local projection for convection--diffusion and
  incompressible flow problems.
\newblock {\em Journal of Scientific Computing}, 43:326--342, 2010.

\bibitem{garg2023local}
Deepika Garg and Sashikumaar Ganesan.
\newblock Local projection stabilized finite element methods for
  advection--reaction problems.
\newblock {\em Calcolo}, 60(4):45, 2023.

\bibitem{gerbeau2006mathematical}
Jean-Fr{\'e}d{\'e}ric Gerbeau, Claude Le~Bris, and Tony Leli{\`e}vre.
\newblock {\em Mathematical Methods for the Magnetohydrodynamics of Liquid
  Metals}.
\newblock Clarendon Press, Oxford, England, 2006.

\bibitem{heumann2010eulerian}
Holger Heumann and Ralf Hiptmair.
\newblock Eulerian and semi-{L}agrangian methods for convection-diffusion for
  differential forms.
\newblock {\em arXiv preprint arXiv:1001.1031}, 2010.

\bibitem{heumann2013stabilized}
Holger Heumann and Ralf Hiptmair.
\newblock Stabilized {G}alerkin methods for magnetic advection.
\newblock {\em ESAIM: Mathematical Modelling and Numerical Analysis},
  47(6):1713--1732, 2013.

\bibitem{heumann2016stabilized}
Holger Heumann, Ralf Hiptmair, and Cecilia Pagliantini.
\newblock Stabilized {G}alerkin for transient advection of differential forms.
\newblock {\em Discrete and Continuous Dynamical Systems-Series S},
  9(1):185--214, 2016.

\bibitem{houston2002discontinuous}
Paul Houston, Christoph Schwab, and Endre S{\"u}li.
\newblock Discontinuous $hp$-finite element methods for
  advection-diffusion-reaction problems.
\newblock {\em SIAM Journal on Numerical Analysis}, 39(6):2133--2163, 2002.

\bibitem{hughes1989new}
Thomas~JR Hughes, Leopoldo~P Franca, and Gregory~M Hulbert.
\newblock A new finite element formulation for computational fluid dynamics:
  Viii. the {G}alerkin/least-squares method for advective-diffusive equations.
\newblock {\em Computer methods in Applied Mechanics and Engineering},
  73(2):173--189, 1989.

\bibitem{hughes1979finite}
Thomas~JR Hughes, Wing~Kam Liu, and Alec Brooks.
\newblock Finite element analysis of incompressible viscous flows by the
  penalty function formulation.
\newblock {\em Journal of Computational Physics}, 30(1):1--60, 1979.

\bibitem{knobloch2010generalization}
Petr Knobloch.
\newblock A generalization of the local projection stabilization for
  convection-diffusion-reaction equations.
\newblock {\em SIAM Journal on Numerical Analysis}, 48(2):659--680, 2010.

\bibitem{knobloch2009local}
Petr Knobloch and Gert Lube.
\newblock Local projection stabilization for advection--diffusion--reaction
  problems: One-level vs. two-level approach.
\newblock {\em Applied Numerical Mathematics}, 59(12):2891--2907, 2009.

\bibitem{matthies2007unified}
G.~Matthies, P.~Skrzypacz, and L.~Tobiska.
\newblock A unified convergence analysis for local projection stabilisations
  applied to the {O}seen problem.
\newblock {\em {ESAIM}: Mathematical Modelling and Numerical Analysis},
  41(4):713--742, 2007.

\bibitem{mizukami1985petrov}
Akira Mizukami and Thomas~JR Hughes.
\newblock A {P}etrov-{G}alerkin finite element method for convection-dominated
  flows: an accurate upwinding technique for satisfying the maximum principle.
\newblock {\em Computer Methods in Applied Mechanics and Engineering},
  50(2):181--193, 1985.

\bibitem{wang2023exponentially}
Jindong Wang and Shuonan Wu.
\newblock Exponentially-fitted finite elements for {$H$(curl) and $H$(div)
  }convection-diffusion problems.
\newblock {\em arXiv preprint arXiv:2308.07680}, 2023.

\bibitem{wang2024discontinuous}
Jindong Wang and Shuonan Wu.
\newblock Discontinuous {G}alerkin methods for magnetic advection-diffusion
  problems.
\newblock {\em Computers and Mathematics with Applications}, 174:43--54, 2024.

\bibitem{wang2024hybridizable}
Jindong Wang and Shuonan Wu.
\newblock A hybridizable discontinuous galerkin method for magnetic
  advection--diffusion problems.
\newblock {\em Journal of Scientific Computing}, 99(3):86, 2024.

\bibitem{wu2020simplex}
Shuonan Wu and Jinchao Xu.
\newblock Simplex-averaged finite element methods for {$H$(grad), $H$(curl),
  and $H$(div)} convection-diffusion problems.
\newblock {\em SIAM Journal on Numerical Analysis}, 58(1):884--906, 2020.

\end{thebibliography}

\end{document}